\documentclass[12pt]{article}
\title{}
\author{}
\usepackage{amsmath,amsthm,amsfonts,amssymb,natbib,graphicx}
\usepackage[margin=1.1in]{geometry}
\usepackage{xcolor}
\usepackage{rotating}
\newtheorem*{theorem*}{Theorem}
\newtheorem{theorem}{Theorem}
\newtheorem{lemma}{Lemma}

\newtheorem{remark}{Remark}

\DeclareMathOperator*{\argmin}{arg\,min}

\newcommand{\eqnref}[1]{\eqref{eqn:#1}}
\newcommand{\lemref}[1]{Lemma~\ref{lem:#1}}
\newcommand{\thmref}[1]{Theorem~\ref{thm:#1}}
\newcommand{\secref}[1]{Section~\ref{sec:#1}}
\newcommand{\appref}[1]{Appendix~\ref{app:#1}}

\newcommand{\R}{\mathbb{R}}
\newcommand{\f}{\mathsf{f}}

\newcommand{\g}{\mathsf{g}}

\newcommand{\inner}[2]{\langle #1 , #2 \rangle}
\newcommand{\norm}[1]{\| #1 \|}
\newcommand{\fronorm}[1]{\| #1 \|_{\textnormal{F}}}
\newcommand{\Fronorm}[1]{\left\| #1 \right\|_{\textnormal{F}}}

\newcommand{\rank}{\textnormal{rank}}

\newcommand{\Proj}{\mathcal{P}}

\newcommand{\Xh}{\widehat{X}}

\newcommand{\pr}[1]{\mathcal{P}_{{#1}}}

\newcommand{\ee}[1]{\mathbf{e}_{{#1}}}
\newcommand{\ident}[1]{\mathbf{I}_{{#1}}}

\newcommand{\ones}{\mathbf{1}}

\newcommand{\Xs}{X_\star}
\newcommand{\As}{A_\star}
\newcommand{\Bs}{B_\star}
\newcommand{\Us}{U_\star}
\newcommand{\Vs}{V_\star}
\newcommand{\Sigmas}{\Sigma_\star}
\renewcommand{\Pr}[2]{\mathcal{P}_{{#1}}\left({#2}\right)} 
\newcommand{\tPr}[2]{\widetilde{\mathcal{P}}_{{#1}}\left({#2}\right)} 

\newcommand{\Ds}{D_\star}
\newcommand{\Ss}{S_\star}
\newcommand{\Sset}{\mathcal{S}}

\newcommand{\footremember}[2]{%
    \footnote{#2}
    \newcounter{#1}
    \setcounter{#1}{\value{footnote}}%
}
\newcommand{\footrecall}[1]{%
    \footnotemark[\value{#1}]%
}

\title{An equivalence between critical points for rank constraints versus low-rank factorizations}
\author{\normalsize Wooseok Ha\footremember{B}{Department of Statistics, University of California, Berkeley},
Haoyang Liu\footremember{C}{Department of Statistics, University of Chicago},
Rina Foygel Barber\footrecall{C}}
\date{}
\begin{document}
	
	\maketitle
	
	\begin{abstract}
Two common approaches in low-rank optimization problems are either working directly with a rank constraint on the matrix variable,
or optimizing over a low-rank factorization so that the rank constraint is implicitly ensured.
In this paper, we study the natural connection between the rank-constrained and factorized approaches. 
We show that all second-order stationary points of the factorized objective function correspond to fixed points of projected gradient descent run on the original problem
(where the projection step enforces the rank constraint). This result allows us to unify many existing optimization guarantees
that have been proved specifically in either the rank-constrained or the factorized setting,
and leads to new results for certain settings of the problem. We demonstrate application of our results to several concrete low-rank optimization problems arising in matrix inverse problems.
	\end{abstract}
	
\section{Introduction}

We consider the following low rank optimization problem
\begin{equation}\label{eqn:low_rank_opt}
	\min_{X\in\R^{m\times n}} \big\{\f(X)   : \rank(X)\leq r\big\},
\end{equation}
for a differentiable function $\f:\R^{m\times n}\to \R$. Due to a wide range of applications, this type of optimization problem has been studied extensively in the past decade. 

In some special cases, the unconstrained minimizer of $\f(X)$ may already be low-rank, i.e.
\begin{equation*}\label{eqn:low_rank_unconstr}
	\Xh \in \argmin_{X\in\R^{m\times n}}\big\{\f(X)   : \rank(X)\leq r\big\} \subseteq \argmin_{X\in\R^{m\times n}} \f(X) . \end{equation*}
This setting arises naturally in the  matrix inverse problems, such as  matrix sensing \citep{recht2010guaranteed} and  matrix completion \citep{candes2009exact}, where the low-rank solution typically represents a matrix signal to recover from a fewer number of  measurements.  In these settings, while there may exist many full rank minimizers due to the nature of under-determined system,  enforcing the constraint over the course of an iterative algorithm  allows to accurately find the one with low rank \citep{oymak2018sharp}. A low-rank solution to the unconstrained minimization problem can also arise in the study of semidefinite programs (SDP)---a wide class of SDP problems\footnote{While canonical forms of SDPs involve linear constraints and do not fall within the framework of~\eqnref{low_rank_opt}, here we mainly focus on the penalized formulation of SDPs, as proposed in~\citep{bhojanapalli2018smoothed}, i.e., the linear constraints are replaced by  a quadratic penalty in the objective function---see~\secref{related_penSDP}. }  admit low rank solution that are global optimal (e.g., \citet{bhojanapalli2018smoothed}).  While SDP problems are convex and can be solved by convex optimization algorithms, restricting the search space via rank constraint may still be useful in speeding up the algorithm~\citep{burer2003nonlinear}. 

In other settings, the rank constraint $\rank(X)\leq r$ will be active in the solution to the minimization problem~\eqnref{low_rank_opt}, meaning
that the unconstrained minimizer will no longer be low rank and we must necessarily work with the rank constraint in the optimization.
In this case, two of the most common optimization strategies in the literature are: either working with the full variable $X\in\R^{m\times n}$ while enforcing $\rank(X)\leq r$ (e.g., by projecting to this constraint
after each iteration), or reformulating the problem in terms of a factorization
$X=AB^\top$ with $A\in\R^{m\times r}$ and $B\in\R^{n\times r}$, so that the factorization ensures the rank constraint. (Riemannian optimization~\citep{absil2009optimization,vandereycken2013low,mishra2013low} is another well-studied approach to optimizaiton under rank constraints which we do not consider in this work. There is also extensive literature on relaxing rank constraint to a convex penalty or constraint, such as the nuclear norm \citep{recht2010guaranteed}, but
here we will focus on optimization techniques that work with the original rank constraint rather than a relaxation.)

Working either with $X$ or with a factorization, we can
implement various optimization  methods  to attempt to find the solution to~\eqnref{low_rank_opt}. When working with the full variable, a standard approach is to treat the rank-constrained set as a subset of the Euclidean space $\R^{m\times n}$, and apply constrained optimization algorithms.
As our central example of this work, we  consider the projected gradient descent method (also known as iterative hard thresholding, see \citet{jain2014iterative}):
\begin{equation}\label{eqn:PGD_intro}X \leftarrow \Proj_r\big(X - \eta\nabla\f(X)\big),\end{equation}
where $\Proj_r(\cdot)$ denotes projection to the rank-$r$ constraint (calculated by taking the top $r$ components of a singular value decomposition).
On the other hand, if we work instead in the factorized setting, we would aim to solve
\begin{equation}\label{eqn:factored_intro}\min_{A\in\R^{m\times r},B\in\R^{n\times r}}\f(AB^\top).\end{equation}
For instance, we might apply any unconstrained optimization techniques to this minimization, which attempt to update  each of the two factors $A, B$. 
In contrast to the full-dimensional approach, these methods implicitly explore the space of low rank matrix manifold  embedded in $\R^{m\times n}$.

Comparing these options naturally raises the following question: is there a connection between the output of full-dimensional approaches such as PGD~\eqnref{PGD_intro} versus
factorized approaches aiming to solve~\eqnref{factored_intro}?
Our work is intended to partially answer this question, and further highlighting the implication of this result to a range of low rank estimation problems.

\subsection{Comparing full-dimensional vs factorized approaches} 

In this work we strengthen the connection between solving the rank-constrained optimization problem via its factorized representation~\eqnref{factored_intro} versus projecting directly to the constraint~\eqnref{PGD_intro}. 
Our key finding is that these two approaches, treated more or less separately in the literature, can in fact be considered to be equivalent for a wide class of low-rank optimization problems, and thus lead to the same guarantees in a range of settings.
Specifically we can state our main result as follows:
\begin{quote}
	Any second-order stationary point (SOSP) of the factorized objective function $\g(A,B) = \f(AB^\top)$, must also be a fixed point of projected gradient descent on the original objective function $\f(X)$.
\end{quote}
\noindent Based on this finding, we further verify the following results:
\begin{itemize}
	\item In \secref{conv}, under conditions of restricted strong convexity/smoothness on  $\f$, we give a range of different
	optimality guarantees for SOSPs of the factorized objective function. Here the strength of the guarantee (e.g., global or local
	optimality) varies depending on the strength of our assumptions on problem.
	\item In \secref{appl}, we specialize these optimality guarantees to several concrete matrix inverse problems arising in low-rank signal recovery, such as matrix sensing, matrix completion, and robust PCA. 
\end{itemize}
As we will see, these results directly follow from our main equivalence result, in combination with some properties of fixed points of PGD~\eqnref{PGD_intro}. 
It is not the aim of this work to provide novel guarantees for estimation and convergence of these various problems---and indeed,
some of these guarantees are already known in the literature, although in other cases new guarantees arise as a byproduct of our main results.
Rather, we aim to bring in a new perspective and broaden our understanding on the landscape of nonconvex low-rank minimization problems through our equivalence result.

\subsection{Notation}
Throughout the paper, $\f:\R^{m\times n}\rightarrow\R$ is a twice-differentiable objective function. Its gradient $\nabla \f(X)$ is represented as a matrix in $\R^{m\times n}$ while its second derivative $\nabla^2\f(X):\R^{m\times n}\times \R^{m\times n}\rightarrow\R$ will be written as a quadratic form, i.e., $\nabla^2\f(X)\big(X_1,X_2\big)$.

We will work also with $\g(A,B) = \f(AB^\top)$, the function defining the factorized problem. Writing $\g:\R^{m\times r}\times \R^{n\times r}\rightarrow \R$, the first derivative $\nabla\g(A,B) = \big(\nabla_{\!A\,}\g(A,B),\nabla_{\!B\,}\g(A,B)\big)$ lies in $\R^{m\times r}\times \R^{n\times r}$, while the second derivative $\nabla^2\g(A,B)$ is a  quadratic form mapping from $\big(\R^{m\times r}\times \R^{n\times r}\big)\times\big(\R^{m\times r}\times \R^{n\times r}\big)$ to $\R$.

For a matrix $X$, we write, respectively, $\fronorm{X}$ and $\norm{X}$ to denote the Frobenius norm and the spectral norm, while  $\norm{X}_{2,\infty}$ will be denoted as the largest $\ell_2$ norm of any row. The $\ell_0$ norm, $\norm{\cdot}_0$, will denote the number of nonzero entries in a vector.
If $\rank(X)\leq r$, we will write
$X=U_X\cdot \textnormal{diag}\{\sigma_1,\dots,\sigma_r\}\cdot V_X^\top$ to denote a (possibly non-unique) singular value decomposition
of $X$, with $\sigma_1\geq \dots\geq \sigma_r$.

\section{Main result}\label{sec:main_result}
We now turn to our main result, relating critical points of factorized optimization of $\g(A,B)=\f(AB^\top)$ to 
the fixed points of PGD on the full-dimensional problem $\f(X)$.
Before proceeding, we need one additional piece of notation
that allow us to quantify the smoothness of $\f$ on the space of 
low-rank matrices:
\begin{equation}\label{eqn:local_smth}\beta_{\mathsf{local}}(X) = \lim_{\epsilon\rightarrow 0}\left\{\sup_{\substack{0<\fronorm{Y-X}\leq \epsilon\\\rank(Y)\leq r}} \frac{\f(Y) - \f(X) - \inner{\nabla\f(X)}{Y-X}}{\frac{1}{2}\fronorm{X-Y}^2}\right\}.\end{equation}
Note that, if $\f$ is twice differentiable, then $\beta_{\mathsf{local}}(X) \leq \norm{\nabla^2\f(X)}$.

This local curvature measure will relate to the step size of PGD, since the step size for PGD is typically
chosen with respect to the curvature of $\f$---in particular, if the second derivative of $\f$ is globally
bounded by some $\beta$, then a constant step size $\eta \leq 1/\beta$ ensures that each step of PGD
will make progress towards minimizing $\f$.

\subsection{Preliminaries: characterizing critical points}\label{sec:main_prelim}
We begin by characterizing a critical point (CP) or fixed point for each of the relevant
representations and algorithms.

\subsubsection{Critical points, fixed points, and local minima of rank-constrained minimization}

First consider the rank-constrained minimization problem~\eqnref{low_rank_opt} over the full-dimensional matrix variable $X$.
For a matrix $X$ with $\rank(X)\leq r$,
\begin{equation}\label{eqn:CP_fulldim}
	\textnormal{$X$ is a CP of~\eqnref{low_rank_opt} 
		iff }\begin{cases}\textnormal{$\rank(X)=r$, $\nabla\f(X)^\top U_X = 0$, and $\nabla\f(X) V_X = 0$, or}\\
		\textnormal{$\rank(X)<r$ and $\nabla\f(X)=0$.}\end{cases}\end{equation}
These conditions are necessary for local optimality \citet[Theorem 6.12]{rockafellar2009variational}---that is, any local minimum $X$
for the function $\f(X)$ must satisfy~\eqref{eqn:CP_fulldim}---but in general are not sufficient. In particular, 
we can verify the following stronger property necessary for local optimality:
\begin{lemma}\label{lem:localmin_vs_fixedpt}
	Suppose that $X$ is a local minimum of the rank-constrained optimization problem~\eqref{eqn:low_rank_opt}. Then,
	in addition to the first-order conditions~\eqref{eqn:CP_fulldim}, the gradient $\nabla\f(X)$ satisfies
	\begin{equation}\label{eqn:localmin_gradbound}\norm{\nabla\f(X)}\leq \beta_{\mathsf{local}}(X) \cdot \sigma_r,\end{equation}
	where $\sigma_r$ is the $r$-th singular value of $X$.
\end{lemma}

Next we turn to the PGD algorithm in particular, and characterize its fixed points. 
Recall that the PGD algorithm has update steps of the form
\[X_{t+1} \leftarrow \pr{r}\big(X_t - \eta \nabla \f(X_t)\big),\]
where $\eta>0$ is the step size, while $\pr{r}$ denotes (possibly non-unique) projection to the rank constraint, i.e., $\pr{r}(X) \in \arg\min_{\rank(X')\leq r}\fronorm{X'-X}$.

A matrix $X\in\R^{m\times n}$ is therefore a {\em fixed point} of PGD at step size $\eta>0$ if it satisfies\footnote{If the projection step
	is not unique, we need to be more precise with our definition. 
	We say that $X$ is a fixed point of PGD at step size $\eta$ if $X$ is equal to a (possibly non-unique) solution
	of the projection step, i.e., $X\in\arg\min_{\rank(X')\leq r}\fronorm{X' - \big(X - \eta\nabla\f(X)\big)}$.}
\[X =\pr{r}\big(X - \eta\nabla\f(X)\big).\]
By examining this condition, we can easily determine that $X$ is a fixed point of PGD if and only if
\begin{equation}\label{eqn:PGD_SP}
	\nabla \f(X)^\top U_X = 0\textnormal{ and }\nabla \f(X) V_X = 0\textnormal{ and } \eta \norm{\nabla \f(X)} \leq \sigma_r.\end{equation}
Comparing to the result of \lemref{localmin_vs_fixedpt} and the critical point conditions~\eqref{eqn:CP_fulldim}, we see that this implies
{\footnotesize
	\begin{equation*}\label{eqn:main_result}
		\left\{\begin{tabular}{@{}c@{}}
			\textnormal{Local minima}\\\textnormal{of \!\!\!\!$\displaystyle\min_{\rank(X)\leq r}\!\!\f(X)$}\end{tabular}\right\}\subseteq 
		\left\{\begin{tabular}{@{}c@{}}
			\textnormal{Fixed pts.~of PGD}\\\textnormal{on $\displaystyle\min_{\rank(X)\leq r}\!\!\f(X)$}\\\textnormal{with $\eta\leq 1/\beta_{\mathsf{local}}$}\end{tabular}\right\}\subseteq 
		\left\{\begin{tabular}{@{}c@{}}
			\textnormal{Critical pts.}\\\textnormal{of \!\!\!\!$\displaystyle\min_{\rank(X)\leq r}\!\!\f(X)$}\end{tabular}\right\}.\end{equation*}
}

\subsubsection{Critical points of factorized minimization}
Next, we will consider the critical points of the factorized objective function $\g(A,B)$, 
defined over the variables $A\in\R^{m\times r}$ and $B\in\R^{n\times r}$ (with no constraints on these variables).
A first-order stationary point (FOSP), or critical point, of $\g$ is any pair $(A,B)$ with $\nabla \g(A,B)=(\nabla_A \g(A,B), \nabla_B \g(A,B))=0$. By definition of $\g$, we can calculate
\[\nabla_{\!A\,}\g(A,B) = \nabla\f(AB^\top) B\textnormal{ and }\nabla_{\!B\,}\g(A,B) = \nabla\f(AB^\top)^\top A,\]
and therefore,
\begin{multline}\label{eqn:g_FOSP}
	\textnormal{The pair $(A,B)$ is a FOSP of $\g$ iff $\nabla\g(A,B)=0$,}\\\textnormal{or equivalently, $\nabla\f(AB^\top)^\top A= 0$ and $\nabla \f(AB^\top) B = 0$.}\end{multline}
Comparing to the first-order optimality conditions for the original (full-dimensional) objective function $\f(X)$, given in~\eqnref{CP_fulldim},
we obtain the following result (which requires no proof):
\begin{lemma}\label{lem:FOSP_compare}
	Let $(A,B)\in\R^{m\times r}\times \R^{n\times r}$. If $X=AB^\top$ 
	is a critical point of $\min_{\rank(X)\leq r}\f(X)$, then
	the pair $(A,B)$ is a FOSP of the  factorized objective function $\g(A,B)$.
\end{lemma}
However, we cannot hope for the converse to be true, since FOSPs of $\g$ can exhibit some counterintuitive behavior that does not arise
in the full-dimensional problem. A well-known example is the pair $(A,B)=(\mathbf{0}_{m\times r},\mathbf{0}_{n\times r})$. This point is always
a FOSP of the factorized problem, but in general $X=\mathbf{0}_{m\times n}$ does not correspond to a critical point of $\f(X)$
(and indeed, will be far from optimal). From this trivial example, we see that considering only the first-order conditions of $\g$ is not sufficient to understand
the correspondence between the full-dimensional and the factorized forms of the problem.
We will therefore next consider second-order stationary points (SOSPs), or critical points, of the factorized problem, which are characterized by the conditions
\begin{equation}\label{eqn:g_SOSP}\nabla \g(A,B)=0 \textnormal{ and } \nabla^2\g(A,B)\succeq 0. \end{equation}

\subsection{Characterization of SOSP for factorized problem}

From the discussion above, we see clearly that any fixed point of the PGD is first-order stationary point (FOSP) of the factorized objective function. Our main theoretical result establishes a partial converse to this,
proving that any second-order stationary point (SOSP) of the factorized objective function $\g(A,B)$ must also be 
a fixed point of projected gradient descent on the original function $\f(X)$.
\begin{theorem}\label{thm:main}Assume that $\f$ is twice differentiable,
	and let $(A,B)\in\R^{m\times r}\times \R^{n\times r}$.
	\begin{enumerate}
		\item[(a)] If $(A,B)$ is a SOSP of the factorized objective function $\g(A,B)$,
		then $X=AB^\top$ is a fixed point of the projected gradient
		descent algorithm on $\min_{\rank(X)\leq r}\f(X)$ with any step size $\eta\leq 1/\beta_{\mathsf{local}}(X)$. 
		\item[(b)]
		Conversely, if $(A,B)$ is not a SOSP of $\g$, then $X=AB^\top$ is not a local minimum of $\min_{\rank(X)\leq r}\f(X)$.
	\end{enumerate}
\end{theorem}
\noindent To summarize, our main result (combined with the discussion of \secref{main_prelim}) shows that,
for the case of a twice-differentiable function $\f$, we have:
{\footnotesize
	\begin{equation*}
		\left\{\begin{tabular}{@{}c@{}}
			\textnormal{Local minima}\\\textnormal{of \!\!\!\!$\displaystyle\min_{\rank(X)\leq r}\!\!\f(X)$}\end{tabular}\right\}\subseteq 
		\left\{\begin{tabular}{@{}c@{}}
			\textnormal{$AB^\top$ for SOSPs}\\\textnormal{$(A,B)$ of $\g(A,B)$}\end{tabular}\right\}\subseteq 
		\left\{\begin{tabular}{@{}c@{}}
			\textnormal{Fixed pts.~of PGD}\\\textnormal{on $\displaystyle\min_{\rank(X)\leq r}\!\!\f(X)$}\\\textnormal{with $\eta\leq 1/\beta_{\mathsf{local}}$}\end{tabular}\right\}\subseteq 
		\left\{\begin{tabular}{@{}c@{}}
			\textnormal{Critical pts.}\\\textnormal{of \!\!\!\!$\displaystyle\min_{\rank(X)\leq r}\!\!\f(X)$}\end{tabular}\right\}\subseteq 
		\left\{\begin{tabular}{@{}c@{}}
			\textnormal{$AB^\top$ for FOSPs}\\\textnormal{$(A,B)$ of $\g(A,B)$}\end{tabular}\right\}.\end{equation*}
}

\subsubsection{Regularized factored optimization}\label{sec:reg_obj}
The factors $A$ and $B$ are not identifiable in the factored optimization problem---in particular, $\g(A,B) = \g(AC,BC^{-1})$ for any invertible $C\in\R^{r\times r}$.
While the product $X=AB^\top$ is in principle not affected by the nonidentifiability of the individual factors, 
it is known that this issue may lead to instability and numerical issues when solving the factorized minimization problem~\eqnref{factored_intro}.
To alleviate this, it is common to add a regularizer on $A$ and $B$  to align the two factors on the same scale (e.g., \citet{tu2015low,Zheng2016Convergence,zhu2018global}). The regularized objective function is
\begin{equation}\label{eqn:reg_factored_loss}\g_{\mathsf{reg}}(A,B)=\g(A,B) + \frac{\lambda }{2}\fronorm{A^\top A - B^\top B}^2, \end{equation}
for a regularization parameter $\lambda>0$. In fact, we can verify that our main result, \thmref{main}, applies in this setting as well.
\begin{lemma}\label{lem:reg_obj}For any $\lambda>0$,
	the result of \thmref{main}(a)
	holds with $\g_{\mathsf{reg}}$ in place of $\g$.  Furthermore, a modification of \thmref{main}(b) holds:
	if $X$ is a local minimum of $\min_{\rank(X)\leq r}\f(X)$, then there exists a factorization $X=AB^\top$ such that $(A,B)$ is a SOSP of $\g_{\mathsf{reg}}$.
\end{lemma}

\subsection{Proof of \thmref{main}}\label{sec:proof_main}
By definition of $\g$, some simple calculations show that $\nabla^2\g(A,B)$ maps $(A_1,B_1) \times (A_1,B_1) $ to
\begin{equation}\label{eqn:g_2nd_deriv}
	2\inner{\nabla \f(X)}{A_1B_1}+ \nabla^2\f(X)\Big(AB_1^\top + A_1B^\top, AB_1^\top + A_1B^\top\Big).\end{equation}

\subsubsection{Claim (a): a SOSP is a fixed point of PGD}
Since we assume that $\nabla^2\g(A,B)\succeq0$ by definition of a SOSP, the calculation in~\eqnref{g_2nd_deriv}
implies that
\begin{equation}\label{eqn:second_deriv_g_psd}
	2\inner{\nabla \f(X)}{A_1B_1^\top}+ \nabla^2\f(X)\Big(AB_1^\top + A_1B^\top, AB_1^\top + A_1B^\top\Big)\geq 0\textnormal{ for all $(A_1,B_1)$}.\end{equation}
By first-order optimality conditions at $(A,B)$ we additionally know that
\begin{equation}\label{eqn:first_deriv_g_zero}
	\nabla\f(X)^\top A = 0\textnormal{ and }\nabla\f(X)B = 0.\end{equation}

Next, let $X=U_X\cdot \textnormal{diag}\{\sigma_1,\dots,\sigma_r\}\cdot V_X^\top$ be a singular value decomposition
of $X$, with $\sigma_1\geq \dots\geq \sigma_r$. Let $u_\star\in\R^m$ and $v_\star\in\R^n$ be the top singular vectors of the gradient $\nabla \f(X)\in\R^{m\times n}$, so that
$\norm{\nabla \f(X)}= u_\star^\top \nabla \f(X) v_\star$.
We will now split into two cases, $\rank(X)=r$ and $\rank(X)<r$.

\paragraph{Case 1: full rank} First suppose  $\rank(X)=r$. Let $u_r$ and $v_r$ be the last left and right singular vectors of $X$, respectively.
Since $X=AB^\top$ has rank $r$, this means that $U_X$ and $A$ span the same column space, and similarly $V_X$ and $B$ span the same column space.
Together with the first-order optimality conditions in~\eqnref{first_deriv_g_zero},
this implies that $\nabla\f(X) V_X=0$ and $\nabla \f(X)^\top U_X=0$.
By our earlier characterization~\eqnref{PGD_SP} of the fixed points of PGD, we therefore only need to check
that $\eta\norm{\nabla \f(X)}\leq \sigma_r(X)$ in order to verify that $X$ is a fixed point of PGD at step size $\eta$. 

Next, if $\nabla\f(X)=0$ then $X$ is obviously a fixed point, so from this point on we will consider the case that $\nabla\f(X)\neq 0$.
Since we know that $\nabla \f(X)^\top U_X =0$ while $u_\star$ is the first left singular vector of $\nabla\f(X)$,
this implies that $u_r^\top u_\star=0$. Similarly $v_r^\top v_\star=0$.
We will consider the curvature of the factorized objective function $\g(A,B)$ in the direction given by
$(A_1,B_1) = \big(-u_\star u_r^\top A,v_\star  v_r^\top B\big)$.
Plugging this choice into our earlier calculation~\eqnref{second_deriv_g_psd} we see that
\begin{multline*}\nabla^2\f(X)\Big(AB_1^\top + A_1B^\top, AB_1^\top + A_1B^\top\Big) \geq -2\inner{\nabla \f(X)}{A_1B_1^\top}\\
	=  2\inner{\nabla \f(X)}{u_\star u_r^\top AB^\top v_r v_\star ^\top} = 2\sigma_r \norm{\nabla \f(X)},\end{multline*}
where the last step holds since $u_r,v_r$ are the $r$th singular vectors of $X=AB^\top$.

Next, we will use the following lemma (proved in \appref{proofs}):
\begin{lemma}\label{lem:RSM_Hessian}
	Let $\f:\R^{m\times n}\rightarrow\R$ be twice-differentiable at $X=AB^\top$, where $A\in\R^{m\times r}$ and $B\in\R^{n\times r}$.
	Then, for any matrices $A_1\in\R^{m\times r},B_1\in\R^{n\times r}$,
	\[\nabla^2\f(X)\big(AB_1^\top + A_1B^\top, AB_1^\top + A_1B^\top\big) \leq \beta_{\mathsf{local}}(X) \cdot \fronorm{AB_1^\top + A_1B^\top}^2.\]
\end{lemma}
Now fix any step size $\eta>0$ with $\eta\leq 1/\beta_{\mathsf{local}}(X)$. Then by \lemref{RSM_Hessian}, along with the definitions of $A_1$ and $B_1$, we can bound
\begin{multline*}
	\nabla^2\f(X)\big(AB_1^\top + A_1B^\top, AB_1^\top + A_1B^\top\big)
	\leq \eta^{-1} \cdot \fronorm{AB_1^\top + A_1B^\top}^2\\
	= \eta^{-1}\cdot \fronorm{AB^\top v_r v_\star^\top - u_\star u_r^\top AB^\top}^2= \eta^{-1}\cdot \sigma_r(X)^2\fronorm{u_r v_\star^\top - u_\star v_r^\top }^2 = 2\eta^{-1}\sigma_r^2,
\end{multline*}
where the next-to-last step holds since $u_r,v_r$ are the $r$th singular vectors of $X=AB^\top$, while the last step holds since $u_r,u_\star$ and $v_r,v_\star$
are pairs of orthogonal unit vectors.
Combining everything, and using the fact that $\sigma_r>0$ since $\rank(X)=r$, we have proved that
\[\eta\norm{\nabla\f(X)}\leq \sigma_r.\]
Applying~\eqnref{PGD_SP}, this verifies that $X$ is a fixed point of PGD with step size $\eta$, which completes the proof for the rank-$r$ case.

\paragraph{Case 2: rank deficient} For the case that $\rank(X)<r$, our proof closely follows that of~\citet[Lemma1]{bhojanapalli2018smoothed}, extending
their result to the asymmetric case (their work assumes $X\succeq 0$ and works with the symmetric factorization $X=AA^\top$).

First, since $A\in\R^{m\times r}$ and $B\in\R^{n\times r}$, if the product $X=AB^\top$ has rank $<r$ then it cannot be the case that both $A$ and $B$
are full rank. Without loss of generality suppose $\rank(A)<r$.
This means that there is some unit vector $w\in\R^r$ with $Aw=0$. Now consider $(A_1,B_1) = (-u_\star w^\top,c\cdot v_\star w^\top)$ for any $c>0$. 
Since $(A,B)$ is a SOSP of the factorized problem, our earlier calculation~\eqnref{second_deriv_g_psd} yields
\[2\inner{\nabla \f(X)}{-c \cdot u_\star w^\top wv_\star^\top}+ \nabla^2\f(X)\Big(c\cdot Aw v_\star^\top + u_\star w^\top B^\top, c\cdot Aw v_\star^\top + u_\star w^\top B^\top\Big)\geq 0.\]
Since $\norm{w}_2=1$ while $u_\star^\top \nabla \f(X) v_\star = \norm{\nabla \f(X)}$, and $Aw=0$ by definition of $w$, we can simplify this to
\[ \nabla^2\f(X)\Big(u_\star w^\top B^\top,u_\star w^\top B^\top\Big)\geq 2c\norm{\nabla \f(X)}.\]
Now, $c>0$ is arbitrary, and so this holds for any $c>0$. On the other hand, since $\f$ is twice-differentiable, the left-hand side must be finite.
This implies that $\norm{\nabla \f(X)}=0$, i.e., $\nabla \f(X)=0$.
Therefore clearly $X$ is a fixed point of projected gradient descent at any step size $\eta$.

\subsubsection{Claim (b): a local minimum is a SOSP}
The second claim follows from a simple Taylor series argument.
Suppose that $X=AB^\top$ is a local minimum of $\min_{\rank(X)\leq r}\f(X)$.
The work in \secref{main_prelim} implies that $(A,B)$ is therefore a FOSP of $\g(A,B)$, that is, $\nabla\g(A,B)=0$.
We therefore only need to verify that $\nabla^2\g(A,B)\succeq 0$ to prove that $(A,B)$ is a SOSP.
Fix any $(A_1,B_1)\in\R^{m\times r}\times\R^{n\times r}$ and let $\delta>0$. Define
\[X_{\delta} = (A + \delta A_1)(B+\delta B_1)^\top.\]
Since $X$ is a local minimum, this implies that $\f(X_{\delta})\geq \f(X)$ for all sufficiently small $\delta>0$.
Next, taking a Taylor expansion,
\begin{align*} 0 &\leq \frac{\f(X_{\delta}) - \f(X)}{\delta^2} = \frac{\inner{\nabla\f(X)}{X_{\delta} - X}}{\delta^2} + \frac{1}{2\delta^2}\nabla^2\f(X)\big(X_{\delta} - X, X_{\delta}-X\big) + \mathcal{O}(\delta)\\
	&=\inner{\nabla\f(X)}{\frac{ AB_1^\top +  A_1B^\top}{\delta} +A_1B_1^\top} + \frac{1}{2}\nabla^2\f(X)\big( AB_1^\top + A_1B^\top,  AB_1^\top +  A_1B^\top\big) + \mathcal{O}(\delta)\\
	&= \inner{\nabla\f(X)}{A_1B_1^\top} + \frac{1}{2}\nabla^2\f(X)\big( AB_1^\top + A_1B^\top,  AB_1^\top +  A_1B^\top\big) + \mathcal{O}(\delta)\\
	&=\nabla^2\g(A,B)\Big((A_1,B_1),(A_1,B_1)\Big) + \mathcal{O}(\delta),
\end{align*}
where the last step applies~\eqnref{g_2nd_deriv}, while the next-to-last step holds since $\nabla\f(X)^\top A = 0$ and $\nabla\f(X) B=0$ due to the fact that $(A,B)$ is a FOSP~\eqnref{g_FOSP}.
Since this bound holds for all sufficiently small $\delta>0$, taking a limit we see that $(A,B)$ is a SOSP.

\subsection{Comparison to related work: penalized SDPs}\label{sec:related_penSDP}

The comparison of full-dimensional approaches versus factorized approaches has been studied in the context of semidefinite programs. 
The existing work closest to the results of our paper is the ``penalized'' form of SDPs~\citep{bhojanapalli2018smoothed}
\[\Xh\in\argmin_{X\in\R^{n\times n},X\succeq 0}\big\{\f_0(X) +\mu \sum_{i=1}^{k} (\f_i(X) - a_i)^2\big\},\]
where $\f_0,\f_1,\dots,\f_k$ are all {\em linear} functions. The factorized form of this problem is given by writing $X=AA^\top$, and solving
\[\min_{A\in\R^{n\times r}}\big\{f_0(AA^\top) +\mu \sum_{i=1}^{k} (\f_i(AA^\top) - a_i)^2\big\}.\]
If we take $r=n$, then the global minimizer of this problem coincides with that of the full SDP---and in fact, this holds as long as $r\geq \rank(\Xh)$.
On the other hand, the factorized problem is  nonconvex so finding the global minimum may be challenging. Remarkably, \citet{bhojanapalli2018smoothed} (building
on the earlier work of \citet{burer2003nonlinear}) show that taking $r\sim \sqrt{k}$ is sufficient to ensure that any second-order stationary point (SOSP) of the factorized
problem is a global minimizer of the full SDP (if one exists); it is also shown that approximate SOSPs are approximately globally optimal. Of course, for $A\in\R^{n\times r}$ to achieve the global minimum at $r\sim \sqrt{k}$, this means that
the global minimizer $\Xh$ itself must have rank on the order of $\sqrt{k}$.

Summarizing, the results mentioned above apply in the setting where:
\begin{itemize}
	\item The optimization problem is a (penalized) SDP, meaning that the functions $\f_0,\f_1,\dots,\f_k$  are linear and the factorized form is given by $X=AA^\top$,
	\item The unconstrained global minimizer $\Xh$ is rank-deficient (without imposing a rank constraint),
	\item Results apply to finding the global minimum.
\end{itemize}
In contrast, in our work, we will allow:
\begin{itemize}
	\item The objective function is any twice-differentiable function $\f(X)$, and $X$ is not necessarily symmetric, i.e., the factorized form is given by $X=AB^\top$,
	\item The unconstrained global minimizer, $\argmin_X\f(X)$, may be full rank in general---in the rank-constrained problem, $\Xh\in\argmin_X\{\f(X):\rank(X)\leq r\}$, the rank constraint may be active,
	\item Results no longer apply to finding the global minimum (since this is NP-hard), but instead we study fixed points.
\end{itemize}
In particular, if a fixed point of the rank-constrained approach is itself a global minimizer, our main result can be made globally, i.e., any SOSP of~\eqnref{factored_intro} is also global optimal. 
In the special case that the problem is a SDP and the SOSP is rank-deficient, i.e., rank strictly less than $r$, this result reduces to the known global optimality result proved by~\citet{bhojanapalli2018smoothed}.\footnote{More precisely, the global solution to the full SDP may not exist in general and \citep{bhojanapalli2018smoothed} proves the result when it exists. To ensure existence of the solution, additional conditions on the linear functions $\f_0,\f_1,\dots,\f_k$ are required; see~\citep{bhojanapalli2018smoothed} for further details. }

\section{Convergence guarantees}\label{sec:conv}

In this section, we investigate the implications of our main result~\thmref{main} on the landscape of the factorized problem~\eqnref{factored_intro}. 
We are interested in determining settings where factorized optimization methods can be expected to achieve
optimality guarantees. Depending on the structure of the objective function $\f$ and other assumptions in the problem, we will
see wide variation in the types of guarantees that can be obtained for the output $\Xh$ of a particular algorithm. From strongest to weakest,
the three main styles of guarantees that appear in the literature are:
\begin{itemize}
	\item Global optimality: the algorithm converges to a global minimizer.
	\item Local optimality, or basin of attraction: if initialized near a global minimizer, then the algorithm converges to that global minimizer.
	\item Restricted optimality: the algorithm converges to a matrix $X$ that satisfies $\f(X)\leq \f(X')$ for any rank-$r'$ matrix $X'$, where $r'<r$ is a strictly lower rank constraint.
\end{itemize}
To simplify our comparison of these three styles of guarantees, we will consider  the setting where the original objective function $\f$ satisfies $\alpha$-{\em restricted strong convexity} (abbreviated as $\alpha$-RSC) with respect to the rank constraint $r$ (\citet{negahban2012unified,agarwal2010fast}), meaning that for all $X,Y\in\R^{m\times n}$ with $\rank(X),\rank(Y)\leq r$,
\begin{equation}\label{eqn:RSC}
	\f(Y)\geq \f(X) + \inner{\nabla \f(X)}{Y-X} + \frac{\alpha}{2}\fronorm{X-Y}^2.
\end{equation}
Similarly, we assume that $\f$ satisfies $\beta$-{\em restricted smoothness}
with parameter $\beta$ (abbreviated as $\beta$-RSM) with respect to the rank constraint $r$, 
meaning that for all $X,Y$ with $\rank(X),\rank(Y)\leq r$, 
\begin{equation}\label{eqn:RSM}
	\f(Y)\leq \f(X) + \inner{\nabla \f(X)}{Y-X} + \frac{\beta}{2}\fronorm{X-Y}^2.
\end{equation}

Throughout this section, we will always write $\kappa = \beta/\alpha$ to denote the rank-restricted condition number of $\f$.
Note that $\kappa\geq 1$ always. We will consider two different regimes for the condition number $\kappa$:
\[\textnormal{Near-isometry ($\kappa \approx 1$) \quad vs. \quad Arbitrary conditioning ($\kappa \gg 1$)}.\]
We can expect to see $\kappa\approx 1$ in certain well-behaved problems, for instance the matrix sensing problem, where
$\f(X)$ represents matching $X$ with random linear measurements of the form $\inner{A_i}{X}$, where e.g., the measurement
matrices $A_i$ have i.i.d.~entries. In general, however, most problems do not have $\kappa\approx 1$.


In some cases, the restricted strong convexity and/or restricted
smoothness conditions might not be satisfied globally (i.e., for all rank-$r$ matrices),
but is satisfied for a more restricted subset of matrices $X,Y$; in these settings we may write, for instance,
that $\f$ satisfies $\alpha$-RSC over a particular subset.

We also need to consider a second important distinction between different classes of problems.
In many statistical settings, we may have an objective function $\f(X)$ that comes from a data likelihood, 
where $\mathbb{E}[\f(X)]$ is minimized at some true low-rank parameter matrix $X_\star$. When this is the case, it is common
to see $\norm{\nabla\f(X_\star)}\approx 0$. In other settings, though, there might not be any natural underlying
low-rank structure, and the gradient $\nabla\f(X)$ is large at any low-rank $X$.
We will therefore distinguish between two scenarios:
\[\textnormal{Vanishing gradient ($\min_{\rank(X)\leq r}\norm{\nabla\f(X)}\approx 0$) \ \  vs. \  Arbitrary gradient ($\min_{\rank(X)\leq r}\norm{\nabla\f(X)}\gg 0$)}.\]

\subsection{Existing results}\label{sec:conv_lit}
We now summarize the existing results as well as our own findings, for the different
types of assumptions and different styles of guarantees outlined above:
\begin{itemize}
	\item Near-isometry + Vanishing gradient $\Rightarrow$ Global optimality.\\
	For the most well-behaved problems, where the objective function $\f(X)$ exhibits
	both near-isometry and a vanishing gradient,
	it is possible to prove convergence to an (approximate) globally optimal estimate $\Xh$. For full-dimensional projected gradient descent algorithm, this has been established in the case of a least squares objective \citep{oymak2018sharp}; for factorized algorithms, an analogous result (no spurious local minima) has been established for certain least squares objectives \citep{bhojanapalli2016global,ge2016matrix,ge2017no,park2016non} and more generally for functions $\f$ with a near-isometry property \citep{zhu2018global}. (We will show in the present work that under near-isometry + vanishing gradient, both full-dimensional and factorized approaches contain no spurious local minima.)


	\item Arbitrary conditioning + Vanishing gradient $\Rightarrow$ Local optimality.\\
	With a non-ideal condition number $\kappa > 1$, assuming a vanishing gradient condition is sufficient 
	to prove a local optimality result,  or the existence of basin of attraction, both for full-dimensional PGD \citep{barber2017gradient}
	and for factorized approaches \citep{chen2015fast}; in the stronger setting of a near-isometry and a vanishing gradient, the local optimality result for factorized approaches has been also established by many works, including \citet{candes2015phase,zheng2015convergent,Zheng2016Convergence,tu2015low,bhojanapalli2016dropping,jain2013low}. Note that all of the previous local optimality results for factorized problems are built upon identifying local region of attraction for globally optimal solution $\Xh$ in the {\em factorized} space $(A,B)$. (We will give in the present work the local region of attraction in the  {\em full-dimensional} representations $X=AB^\top$.)
	
	\item Arbitrary conditioning + Arbitrary gradient $\Rightarrow$ Restricted optimality.\\
	In the most challenging setting, where we allow both arbitrary condition number $\kappa$
	and an arbitrarily large gradient, restricted optimality guarantees can still be obtained. This
	is established for the full-dimensional PGD algorithm \citep{jain2014iterative,liu2018between}, as well as its variants, such as approximate low-rank projection \citep{becker2013randomized,soltani2017fast}, and projection with debiasing step \citep{yuan2018gradient}; for sparse problems specifically, the analogous restricted optimality result has been established \citep{shen2017tight}. On the other hand, there is no known result for restricted optimality guarantees within the factorized approach. (We will show in the present work that it holds also for the factorized approach.)
\end{itemize}

\noindent This extensive literature has enabled us to understand the landscape of the nonconvex low-rank optimization problem, 
but the various results have been proved somewhat disjointly, using very different techniques
for analyzing full-dimensional PGD type algorithms versus factorized algorithms.
It is natural to ask whether this collection of results can be unified into a single framework.
Our main result, \thmref{main}, allows us to connect established results between PGD algorithms and factorized algorithms,
allowing us to establish  simpler proofs of some existing results, and provide new results in certain settings.
Overall, it is the goal of this section to provide a broader view of the landscape
of results known for low-rank optimization problems through the lens of the equivalence between PGD and factorized
algorithms established in~\thmref{main}.

\subsection{Results for global and local optimality}

In the special case of least squares objective, i.e., $\f(X)=\frac{1}{2}\fronorm{\mathcal{A}(X)-b}^2$ for a linear operator $\mathcal{A}:\R^{m\times n}\to \R^p$,  \citet{oymak2018sharp} show that, in the near-isometry setting ($\kappa\approx 1$), projected gradient descent offers a global convergence guarantee starting from any initialization point. Here we extend some of their technical tools to  general functions $\f(X)$.  We will write $\R^{m\times n}_{\rank(r)}$ to denote the set of $m\times n$ matrices with rank $\leq r$.
\begin{lemma}\label{lem:PGD_alphabeta}
	Suppose that $\f:\R^{m\times n}\rightarrow \R$ satisfies $\alpha$-RSC~\eqnref{RSC} over a subset $\mathcal{X}\subseteq\R^{m\times n}_{\rank(r)}$, where $\alpha>0$. 
	If $X_0,X_1\in\mathcal{X}$ are both fixed points of PGD run with step size $\eta_0>0$ or $\eta_1>0$, respectively, then 
	one of the following must hold:
	\begin{itemize}
		\item $X_0=X_1$, or
		\item $\rank(X_0)=r$ and $\rank(X_1)<r$ and $\frac{\norm{\nabla\f(X_0)}}{\sigma_r(X_0)}  \geq 2\alpha$, or
		\item $\rank(X_1)=r$ and $\rank(X_0)<r$ and $\frac{\norm{\nabla\f(X_1)}}{\sigma_r(X_1)}  \geq 2\alpha$, or
		\item $\rank(X_0)=\rank(X_1) = r$ and
		\[ \frac{\norm{\nabla\f(X_0)}}{\sigma_r(X_0)}  + \frac{\norm{\nabla\f(X_1)}}{\sigma_r(X_1)}\geq 2 \alpha.\]
	\end{itemize}  
\end{lemma}
\noindent The proof of this lemma is given in \appref{proofs}. 
We also verify a simple result:
\begin{lemma}\label{lem:globalmin_is_stationary}
	Suppose that $\f:\R^{m\times n}\rightarrow \R$ satisfies $\beta$-RSM~\eqnref{RSM} 
	over an open subset $\mathcal{X}\subseteq\R^{m\times n}_{\rank(r)}$. 
	If $\Xh$ is a global minimizer (i.e., $\f(\Xh)=\min_{\rank(X)\leq r}\f(X)$)  and $\Xh\in\mathcal{X}$, then $\Xh$ is a fixed point of projected gradient descent run with rank constraint $r$ and any step size $\eta\leq 1/\beta$.
\end{lemma}

These lemmas will allow us to easily prove global optimality and local optimality results under the appropriate assumptions.
We now turn to the question of obtaining global and local optimality results for PGD and factorized algorithms. 
While results of this flavor are already known in the literature (see \secref{conv_lit} for some references),
our goal here is to give extremely short and clean proofs that illuminate the connection between the full-dimensional and factorized
representations of the optimization problem, and thereby also highlight the utility of our main result, \thmref{main}. In some cases, our work also establishes guarantees in a broader setting than previous results.

\subsubsection{Global optimality}\label{sec:global_opt}
In the setting where $\f(X)$ satisfies the near-isometry property, with condition number
$\kappa < 2$, we can obtain global optimality guarantees for both PGD and factorized methods
whenever $\norm{\nabla\f(X)}$ is sufficiently small, i.e., the vanishing gradient condition. (See \secref{conv_lit}
for related existing results in the literature.)

\begin{theorem}\label{thm:global_opt}
	Assume that $\f(X)$ satisfies $\alpha$-RSC~\eqnref{RSC}
	and $\beta$-RSM~\eqnref{RSM}  over an open subset $\mathcal{X}\subseteq\R^{m\times n}_{\rank(r)}$, and that $\beta < 2\alpha$. Suppose $\Xh$ is a global minimizer, i.e., $\f(\Xh)=\min_{\rank(X)\leq r}\f(X)$.  If $\Xh\in\mathcal{X}$ and $\Xh$ satisfies
	\[ \text{Either $\rank(\Xh)<r$, or $\rank(\Xh)=r$ and }\norm{\nabla\f(\Xh)}< (2\alpha -\beta)\cdot \sigma_r(\Xh)\;,\]
	then 
	\begin{itemize}
		\item $\Xh$ is the unique fixed point of PGD  in $\mathcal{X}$ for any step size $1/(2\alpha)<\eta\leq 1/\beta$ in the case that $\rank(\Xh)<r$, or in the case $\rank(\Xh)=r$,
		for
		any step size satisfying
		\begin{equation}\label{eqn:global_stepsize}\frac{1}{2\alpha - \frac{\norm{\nabla\f(\Xh)}}{\sigma_r(\Xh)}} < \eta\leq \frac{1}{\beta}.\end{equation}
		\item If $X=AB^\top\in\mathcal{X}$ where $(A,B)$ is a SOSP of $\g(A,B)$, then $X=\Xh$. 
	\end{itemize}
\end{theorem}
\noindent Note that, in the case that $\rank(\Xh)=r$, due to the condition $\norm{\nabla\f(\Xh)}< (2\alpha -\beta)\cdot \sigma_r(\Xh)$  the interval~\eqnref{global_stepsize} given for step size $\eta$ is always non-empty.

\begin{remark} 
	In some examples, such as the matrix sensing problem discussed later in \secref{matrix_sensing},
	the RSC/RSM conditions will hold globally, i.e., for $\mathcal{X} = \R^{m\times n}_{\rank(r)}$; this is why we use the term
	``global optimality'' to describe this result. In other settings,
	the RSC/RSM conditions may not hold universally over all rank-$r$ matrices
	but hold for a subset of matrices, e.g., all matrices satisfying an incoherence condition such as in the robust PCA problem (\secref{rpca});
	the above theorem is formulated to cover this type of scenario as well even though the term ``global optimality'' may no longer apply.
\end{remark}

\thmref{global_opt} proves that global optimality guarantees can be achieved as long as $\kappa<2$, i.e., the map $\f$ is a near-isometry. This type of assumption on $\kappa$ is crucial to achieving global optimality guarantees. For instance, \citet[Example 3]{zhang2018much} construct an example of objective function $\f(X)$ with $\beta=3\alpha$, i.e., $\kappa =3$, where there exists a fixed point $X$ that is {\em not} globally optimal. This proves that $\kappa <3$ is necessary for achieving a global optimality guarantee, while our work shows $\kappa<2$ is sufficient. While it is not the goal of the present work, an interesting open question is to close the gap between these necessary and sufficient conditions to identify an exact correspondence between condition number and the global optimality guarantee; see also~\citet{zhang2019sharp} for the sufficient and necessary conditions when rank $r=1$.

We now compare this result with some recent works in the literature. The first part of~\thmref{global_opt}, i.e., the result for fixed points of PGD  on $X\in\R^{m\times n}$, is an extension of global optimality results established in~\citet{oymak2018sharp}---their work is specific to a least-squares objective function, i.e., $\f$ is quadratic.\footnote{In~\citet{oymak2018sharp}, the authors mention that their results are more broadly applicable than least squares objective functions, but we are not aware of any such results that have appeared in the follow-up papers. } On the other hand, the second part of the theorem, i.e., the result on SOSPs of the factorized problem, is already known for various types of problems, such as the matrix sensing and the matrix completion problems \citep{bhojanapalli2016global,ge2016matrix,ge2017no}. Similarly, \citet{zhu2018global} also establish ``no spurious local minima'' under conditions similar to~\thmref{global_opt}, i.e., when $\f(X)$ satisfies $\alpha$-RSC and $\beta$-RSM with $\alpha\approx \beta$. While these results typically require more involved analysis than our framework presented here, they further prove strict saddle property (see, for instance, \citet[Assumption A2]{jin2017escape}) of the factorized problems under which polynomial time convergence is ensured for finding approximate SOSPs (hence approximate globally optimal solution). Such guarantee on the rate of convergence is not provided in~\thmref{global_opt}, and we leave the study of approximate SOSPs in the future work.

\begin{proof}[Proof of \thmref{global_opt}]
	First we consider PGD. 
	By \lemref{globalmin_is_stationary}, we know that $\Xh$ is a fixed point for any $\eta\leq 1/\beta$.
	
	We first consider the case that $\rank(\Xh)<r$. Let $X\in\mathcal{X}$ be another fixed point of PGD for any step size $1/2\alpha<\eta\leq 1/\beta$.
	Suppose that $X\neq \Xh$. Then applying \lemref{PGD_alphabeta}, we must have $\rank(X) = r$ with $\frac{\norm{\nabla\f(X)}}{\sigma_r(X)}\geq 2\alpha$.
	However, by~\eqnref{PGD_SP}, we have $\norm{\nabla\f(X)}\leq \eta^{-1}\sigma_r(X) < 2\alpha\sigma_r(X)$, which is a contradiction.
	
	Next, consider the case that $\rank(\Xh) = r$, and let $X\in\mathcal{X}$ be another fixed point of PGD for any step size $\eta$
	satisfying~\eqnref{global_stepsize}. Suppose $X\neq \Xh$. By \lemref{PGD_alphabeta}, we either have $\rank(X)<r$ and
	$\frac{\norm{\nabla\f(\Xh)}}{\sigma_r(\Xh)}\geq 2\alpha$, or alternatively $\rank(X)=r$ and
	\[2\alpha \leq \frac{\norm{\nabla\f(\Xh)}}{\sigma_r(\Xh)}  + \frac{\norm{\nabla\f(X)}}{\sigma_r(X)}\leq  \frac{\norm{\nabla\f(\Xh)}}{\sigma_r(\Xh)}  +   \frac{1}{\eta}\]
	by applying~\eqnref{PGD_SP}. In either case, this contradicts our assumption~\eqnref{global_stepsize} on $\eta$.
	
	Next we turn to the factorized setting. Let $X=AB^\top\in\mathcal{X}$ where $(A,B)$ is a SOSP of $\g(A,B)$.
	Comparing the definition of $\beta$-RSM over $\mathcal{X}$ with that of the local smoothness parameter $\beta_{\mathsf{local}}(X)$ defined in~\eqnref{local_smth},
	we can see that since $\mathcal{X}\subseteq\R^{m\times n}_{\rank(r)}$ is an open subset, $\beta_{\mathsf{local}}(X)\leq \beta$ by definition, and therefore $\eta\leq 1/\beta_{\mathsf{local}}(X)$. Therefore,
	applying our main result, \thmref{main}, we see that $X$ must be a fixed point of PGD at step size $\eta = 1/\beta$,
	which proves that $X=\Xh$ by our work above.
\end{proof}

\subsubsection{Local optimality}
Next we turn to the local optimality guarantees, i.e., the existence of local region of attraction, that can be obtained when $\f$ exhibits a vanishing gradient,
but may have an arbitrarily large condition number $\kappa$. (See \secref{conv_lit}
for related existing results in the literature.)

\begin{theorem}\label{thm:local_opt}
	Assume that $\f(X)$ satisfies $\alpha$-RSC~\eqnref{RSC} over a subset $\mathcal{X}\subseteq\R^{m\times n}_{\rank(r)}$. Assume that $\Xh$ is a global minimizer, i.e., $\f(\Xh)=\min_{\rank(X)\leq r}\f(X)$, that $\Xh\in\mathcal{X}$, and that $\Xh$ satisfies
	\[\text{Either $\rank(\Xh)<r$, or $\rank(\Xh)=r$ and }\norm{\nabla\f(\Xh)}< \alpha \cdot \sigma_r(\Xh)\;.\]
	Let
	\[\mathcal{N} = \left\{X\in\mathcal{X} : \rank(X)< r\text{ or }\frac{\norm{\nabla\f(X)}}{\sigma_r(X)} < 2\alpha\right\}\]
	in the case that $\rank(\Xh)<r$, or
	\[\mathcal{N} = \left\{X\in\mathcal{X} :  \rank(X)< r\text{ or }\frac{\norm{\nabla\f(\Xh)}}{\sigma_r(\Xh)} +\frac{\norm{\nabla\f(X)}}{\sigma_r(X)} < 2\alpha\right\}\]
	in the case that $\rank(\Xh)=r$.
	Then:
	\begin{itemize}
		\item  For any fixed point $X$ of PGD with any step size $\eta>0$, if $X\in\mathcal{N}$, then $X=\Xh$. 
		\item If $X=AB^\top\in\mathcal{N}$ where $(A,B)$ is a SOSP of $\g(A,B)$, then $X=\Xh$.
	\end{itemize}
\end{theorem}
\noindent  We note that $\Xh\in\mathcal{N}$ by the assumptions of the theorem. 
In this setting where $\kappa$ may be arbitrarily large, global optimality does not hold in general (as shown by \citet{zhang2018much}'s
counterexample, discussed in \secref{global_opt} above). Nonetheless, the results in~\thmref{local_opt} still ensure the existence of regions of attraction 
$\mathcal{N}$ within which the global minimum $\Xh$ will be discovered, for both the full-dimensional and factorized methods.

To compare with the existing results, the first part of~\thmref{local_opt} (for fixed points of PGD) is an immediate result given the work in~\citet{barber2017gradient}. Next, turning to the second part of the result, on the SOSPs of the factorized approach, some related results in the existing literature have shown that certain rank-constrained problems exhibit local region of attraction near the global minimum $\Xh$ \citep{candes2015phase,zheng2015convergent,Zheng2016Convergence,tu2015low,bhojanapalli2016dropping,jain2013low}. While these problems satisfy the near-isometry property with $\kappa\approx 1$, our result in~\thmref{local_opt} extends to a broader setting with an arbitrarily large condition number $\kappa$. \citet{chen2015fast} have also established local convergence guarantees under conditions similar to restricted strong convexity and smoothness, but the difference is that they work with RSC and RSM type conditions defined directly on the factorized variable pair $(A,B)$. In addition, many of these works address the positive semidefinite setting, $X=AA^\top$, rather than the generic setting $X=AB^\top$ considered here.

\begin{proof}[Proof of \thmref{local_opt}]
	
	By \lemref{localmin_vs_fixedpt}, we know that $\Xh$ is a fixed point for PGD with step size $\eta\leq 1/\beta_{\textsf{local}}(\Xh)$.
	(Note that $\beta_{\textnormal{local}}$ may be arbitrarily large, but must be finite since $\f$ is twice differentiable.)
	Suppose that $X\in\mathcal{N}$ is another fixed point at some  (potentially different) step size $\eta>0$, with $X\neq \Xh$.
	Applying \lemref{PGD_alphabeta} with $X_0 = \Xh$ and $X_1 = X$ yields a contradiction to the definition of $\mathcal{N}$, either for $\rank(\Xh)<r$ or $\rank(\Xh)=r$, unless $X=\Xh$.
	
	Next we turn to factorized setting. 
	By~\thmref{main} we know that any $X=AB^\top\in\mathcal{X}$ for a SOSP $(A,B)$ must be a fixed point of PGD at any step size $\eta\leq 1/\beta_{\textsf{local}}(X)$ (again $\beta_{\textsf{local}}(X)$ can be extremely large).
	If also $X\in\mathcal{N}$ then this 
	proves that $X=\Xh$ by our work above.
	
\end{proof}

\subsection{A restricted optimality guarantee}\label{sec:restricted_opt}

In this last setting, we will make no assumptions on either the gradient or the condition number, i.e., it may be possible that $\norm{\nabla\f(\Xh)}$ is large and the condition $\kappa$ is large as well. (See \secref{conv_lit}
for related existing results in the literature.)

Under such assumptions, to the best of our knowledge, there is no guaranteed result to solve the low-rank minimization problem either locally or globally---identifying a region of attraction in a deterministic way is a nontrivial task. Therefore, we may wish to instead establish a weaker {\em restricted optimality} guarantee, which entails proving that the algorithm converges to some matrix $X$ satisfying
\[\f(X)\leq \min_{\rank(Y)\leq r'}\f(Y),\]
where the rank $r'<r$ proves a more restrictive constraint. In a statistical setting where we are aiming to recover some true low-rank parameter, we might think of $r'$ as the true underlying rank,
while $r\geq r'$ is a relaxed rank constraint that we place on our optimization scheme. More generally, we are simply aiming
to show that optimizing over rank $r$, while not ensuring the best rank-$r$ solution, is competitive with the best lower-rank solution.

Under these conditions, \citet{liu2018between} prove that any fixed point $X$ of PGD with step size $\eta  = 1/\beta$ satisfies
restricted optimality with respect to any rank $r' < r/\kappa^2$.
Based on our main result, \thmref{main}, the same guarantee also holds for any SOSP of the factorized problem.
For completeness, we restate their result along with the new extension to the factorized problem:
\begin{theorem}\label{thm:restricted_opt_upperbd}
	Assume that $\f(X)$ satisfies $\alpha$-RSC~\eqnref{RSC}
	and $\beta$-RSM~\eqnref{RSM}  over an open subset $\mathcal{X}\subseteq\R^{m\times n}_{\rank(r)}$. Let $\kappa = \beta/\alpha$. Then:
	\begin{itemize}
		\item \citep{liu2018between} For any fixed point $X\in\mathcal{X}$ of PGD with step size $\eta = 1/\beta$,
		\begin{equation}\label{eqn:restricted_opt_stationary}\f(X)\leq \min_{\rank(Y)< r/\kappa^2,Y\in\mathcal{X}}\f(Y),\end{equation}
		i.e., $X$ satisfies restricted
		optimality with respect to any rank $r'<r/\kappa^2$ within $\mathcal{X}$.
		\item For any $X=AB^\top\in\mathcal{X}$  where $(A,B)$ is a SOSP of the factorized problem $\g(A,B)$, 
		\[\f(AB^\top) \leq \min_{\rank(Y)< r/\kappa^2,Y\in\mathcal{X}} \f(Y),\]
		i.e., $X=AB^\top$ satisfies restricted
		optimality with respect to any rank $r'<r/\kappa^2$ within $\mathcal{X}$.
	\end{itemize}
\end{theorem}
\begin{proof}[Proof of \thmref{restricted_opt_upperbd}]
	The first claim follows by the definition of  $\alpha$-RSC used on the two points $X,Y$ together with Lemma $1$ in \citet{liu2018between}, which bounds the restricted concavity of hard thresholding. The second claim follows immediately by combining the first claim with \thmref{main},
	as in the proof of \thmref{global_opt}.
\end{proof}

Conversely, \citet{liu2018between} also establish that this factor of $\kappa^2$ is sharp  in general (on all low-rank matrices), i.e., restricted optimality cannot
be guaranteed relative to rank $r'>r/\kappa^2$. Here we establish the analogous result
for the factorized problem. For completeness, we state the two results together.
\begin{theorem}\label{thm:restricted_opt_lowerbd}
	For any parameters $\beta\geq \alpha >0$ and any rank $r' > r/\kappa^2$, 
	there exists a function $\f:\R^{m\times n}\rightarrow \R$ satisfying $\alpha$-RSC~\eqnref{RSC} and $\beta$-RSM~\eqnref{RSM}  over $\R^{m\times n}_{\rank(r)}$,
	such that:
	\begin{itemize}
		\item \citep{liu2018between} There exists a fixed point $X$ of PGD with step size $\eta = 1/\beta$, such that
		\[\f(X) > \min_{\rank(Y)\leq r'} \f(Y).\]
		\item There exists a second-order stationary point  $(A,B)$ of the factorized problem,
		such that
		\[\f(AB^\top) > \min_{\rank(Y)\leq r'} \f(Y).\]
	\end{itemize}
\end{theorem}
\noindent This result is proved in \appref{proofs}. Unlike the restricted optimality guarantee above (\thmref{restricted_opt_upperbd}),
this converse result does not follow directly from \citet{liu2018between}'s work, and instead requires a new construction.

\section{Applications}\label{sec:appl}

In this section we apply our framework developed in~\secref{main_result} and/or~\secref{conv} to several concrete low-rank optimization problems, including matrix sensing, matrix completion, and robust PCA.
These problems typically involve an unknown ground truth matrix $\Xs\in\R^{m\times n}$ that is low-rank and the goal is to accurately recover it from a few or sparse or corrupted measurements. 
In many cases, $\Xs$ itself becomes a global minimizer of the rank-constrained minimization problem~\eqnref{low_rank_opt}  in which case we also denote by $\Xs$ (instead of $\Xh$) to represent a global minimizer.

\subsection{Matrix sensing}\label{sec:matrix_sensing}
In the matrix sensing problem~\citep{recht2010guaranteed}, we aim to recover a low rank matrix $\Xs\in\R^{m\times n}$ given $k$ linear observations $b_1=\inner{L_1}{\Xs},\dots,b_k=\inner{L_k}{\Xs}$. Therefore, the least square objective function $\f: \R^{m\times n}\rightarrow \R$ for the matrix sensing problem takes the form
\begin{equation}\label{eqn:obj_sensing}
	\f(X)=\frac{1}{2k}\sum_{i=1}^{k}\left(\inner{L_i}{X}-b_i\right)^2.
\end{equation}
The corresponding rank-$r$ factorized objective function $\g: \R^{m\times r} \times \R^{n\times r}\rightarrow \R$ is defined as 
\[\g(A,B)=\f(AB^\top)\]
To conform with our notion of condition number, we define the following set of sensing matrices:
\begin{align*}
	&\mathcal{L}(\alpha,\beta,r)=\{(L_1,\dots,L_k):\textnormal{the map $X\mapsto \frac{\sum_{i=1}^k\inner{L_i}{X}^2}{2k}$ is $\alpha$-RSC and $\beta$-RSM}\\
	&\hspace{5cm} \textnormal{with respect to the rank constraint $r$}\}.
\end{align*}
Then direct application of our results in Section~\ref{sec:conv} gives the following lemma (without proof).
\begin{lemma}\label{lem:matrix_sensing}
	Consider a matrix sensing model with a rank-$r$ matrix $\Xs$. Define the objective function $\f(X)$ as in equation~(\ref{eqn:obj_sensing}). Then
	\begin{itemize}
		\item If the sensing matrices satisfy $(L_1,\dots,L_k)\in \mathcal{L}(\alpha,\beta,r)$ with $\beta<2\alpha$, then for any second-order stationary point $(A,B)$ of the factorized objective $\g$, $AB^\top=\Xs$. 
		\item Define the following neighborhood of $\Xs$,
		\[\mathcal{N}(\Xs)=\{X\in \R^{m\times n}_{\rank(r)}: \rank(X)< r \text{ or }  \norm{\nabla\f(X)}< 2\alpha\cdot \sigma_r(X)\}.\]
		If the sensing matrices satisfy $(L_1,\dots,L_k)\in \mathcal{L}(\alpha,\beta,r)$, then for any second-order stationary point $(A,B)$ of the factorized objective $\g$, if $X=AB^\top \in \mathcal{N}(\Xs)$ then $X=\Xs$.
		\item If the sensing matrices  satisfy $(L_1,\dots,L_k)\in \mathcal{L}(\alpha',\beta',r')$ with $r'> (\frac{\beta'}{\alpha'})^2r$, then for any second-order stationary point $(A,B)$ of the  rank-$r'$ factorized objective function $\g_{r'}$, $\g_{r'}(A,B)=0$.\footnote{The rank-$r'$ objective $\g_{r'}$ is defined as the function $\g_{r'}(A,B)=\f(AB^\top)$ where the factorization $X=AB^\top$ is over-parametrized by rank $r'>r$, i.e. $A\in\R^{m\times r'}$ and $B\in\R^{n\times r'}$.} Since $\f$ satisfies $\alpha'$-RSC~ over $\R^{m\times n}_{\rank(r')}$ and $\f(\Xs)=0$ while $\nabla\f(\Xs)=0$, this further implies that $AB^\top=\Xs$.
	\end{itemize}
\end{lemma}

\noindent In the simplest setting, the sensing matrices  $(L_1,\dots,L_k)$ are drawn i.i.d. from standard normal distribution $\mathcal{N}(0,1)$, then, with high probability, $(L_1,\dots,L_k)\in \mathcal{L}(\alpha,\beta,r)$ with $\beta<2\alpha$ (e.g., \citet{recht2010guaranteed}). In this case  global optimality of SOSPs of $\g$ follows from~\lemref{matrix_sensing} in a straightforward manner. In the more general setting where  the sensing matrices  $(L_1,\dots,L_k)$ are drawn i.i.d. from normal distribution $\mathcal{N}(0,\Sigma)$ with covariance matrix $\Sigma\in\R^{mn\times mn}$, \citet[Lemma 7]{agarwal2010fast} proves that with high probability $(L_1,\dots,L_k)\in \mathcal{L}(\alpha,\beta,r)$ with $\alpha=c_1\lambda_{\mathsf{min}}(\Sigma)$ and $\beta=c_2\lambda_{\mathsf{max}}(\Sigma)$ for some $c_1,c_2>0$. In this case, \lemref{matrix_sensing} provides respectively local and restricted optimality guarantees for SOSPs of the factorized problem. In particular, we observe that  any SOSPs $(A,B)$ of the factorized problem still achieve the global minimizer, i.e. $AB^\top = \Xs$, if we over-parametrize $\f$ with rank $r'\approx \lambda^2_{\mathsf{max}}(\Sigma)/\lambda^2_{\mathsf{min}}(\Sigma)\cdot r$. This result has not been known previously in the literature of matrix sensing but somewhat follows directly from our main correspondence result~\thmref{main}.

\subsection{Robust PCA}\label{sec:rpca}

Robust PCA \citep{candes2011robust,chandrasekaran2011rank} refers to the decomposition of the data matrix $\Ds$ into a low-rank component $\Xs$ and a sparse component $\Ss$, so that the sum of two components recover the original matrix. One can view the sparse component to be some  outliers which we wish to separate from the low-rank signal. Concretely, suppose that we are given $\Ds=\Xs+\Ss \in \R^{m\times n}$ with $\rank(\Xs)\leq r$ and where $\Ss$ is $s$-sparse in each column, then we consider the following minimization problem
\begin{equation}\label{eqn:robust_PCA_problem}
	\min_{X}  \left\{\f(X) = \frac{1}{2}\min_{S\in\Sset} \fronorm{\Ds - (X+S)}^2:\rank(X)\leq r \right\}.
\end{equation}
Here to specify the sparsity of the sparse component $S$, we set 
\[\Sset=\{S\in\R^{m\times n}: \norm{S_{j}}_0\leq  \norm{\Ss{}_{j}}_0=s \text{ for } j=1,\ldots,n\}, \]
where $S_j$ and $\Ss{}_{j}$ denote $j$-th columns of $S$ and $\Ss$ respectively. 

Before we apply our  results to the above problem, note that the factorized objective function $\g(A,B)=\f(AB^\top)$ in~\eqnref{robust_PCA_problem} is not twice-differentiable with respect to $(A,B)$. To extend our discussion on this setting, at any point $X$ with rank $\leq r$, we consider the following majorization function of $\f$:
\[\f_X(\widetilde{X}) =\frac{1}{2}\fronorm{\Ds - (\widetilde{X}+S(X))}^2, \]
where we define $S(X)\in\argmin_{S\in\Sset}\fronorm{\Ds - (X+S)}^2$.
Then it is easy to see that $\f_X(\widetilde{X})$ majorizes the original function $\f$ while matches at $X$ up to the first-order term, i.e. $\f_X(\widetilde{X})\geq \f(\widetilde{X})$ and $\f_X(X)=\f(X)$ and $\nabla\f_X(X)=\nabla\f(X)$. Denoting $\g_{X}(\widetilde{A},\widetilde{B})=\f_X(\widetilde{A} \widetilde{B}^\top)$ to be the factorized form of $\f_X$, we utilize the following version of second-order stationary points of $\g$, which is defined as (see also~\citet[Definition 5]{ge2017no})
\begin{equation}\label{eqn:SOSP_rpca}\nabla\g(A,B)=0 \textnormal{ and } \nabla^2\g_{X}(A,B)\succeq 0. \end{equation}
That is, we say $(A,B)$ is a second-order stationary point of $\g$ if it satisfies the conditions~\eqnref{SOSP_rpca} with $X=AB^\top$.

Next suppose that $\Xs=\As\Bs^\top$, where $\As=\Us\sqrt{\Sigmas}$ and $\Bs=\Vs\sqrt{\Sigmas}$, and $\Xs=\Us\Sigmas\Vs^\top$ be the rank-$r$ SVD of $\Xs$.
It is well-known that the robust PCA model suffers from non-identifiability issue and in particular we cannot recover the true matrix $\Xs$ if $\Xs$ is itself both low-rank and {\em sparse}. To prevent this, we assume that $\Xs$ is incoherent relative to the sparse matrices~\citep{candes2011robust}, i.e., 
\begin{equation}\label{eqn:mu_incoherence}\norm{\As}_{2,\infty}\leq \sqrt{\sigma_1(\Xs)\frac{\mu r}{m}}, \;\;\norm{\Bs}_{2,\infty}\leq \sqrt{\sigma_1(\Xs)\frac{\mu r}{n}}.\end{equation}

With this definition in place,  now suppose that $(A,B)$ is a SOSP of $\g$ with $X=AB^\top$. Since $\f_X(\cdot)$ trivially satisfies RSC~\eqnref{RSC} and RSM~\eqnref{RSM} with $\alpha=\beta=1$ (over the entire space of matrices $\R^{m\times n}$), and at a global minimum $\widetilde{X}_{\mathsf{global}}$, i.e., $\f(\widetilde{X}_{\mathsf{global}})=\min_{\rank(\widetilde{X})\leq r}\f_X(\widetilde{X})$,
\[\norm{\nabla\f_X(\widetilde{X}_{\mathsf{global}})}=\sigma_{r+1}(\Ds - S(X)) < \sigma_r(\Ds-S(X))=(2\alpha-\beta)\sigma_r(\widetilde{X}_{\mathsf{global}}),\footnote{It can be shown that the pathological case, $\sigma_r(\Ds-S(X))=\sigma_{r+1}(\Ds-S(X))$, does not occur under an additional assumption on the magnitude of the true sparse matrix $\Ss$, see~\appref{rpca} for further details; see also~\citet{ge2017no}. To simplify the presentation, here we assume this is the case.}\]
our result~\thmref{global_opt} applies to show that $AB^\top=\widetilde{X}_{\mathsf{global}}$ is the unique fixed point of PGD run on $\f_X(\cdot)$ at step size $\eta=1/\beta=1$. In other words, $X =  \Proj_r\big(\Ds - S(X) \big)$. 

Furthermore, by definition of $S(X)$, this means that $(X,S(X))$ is together a joint fixed point of PGD with step sizes $\eta_X=\eta_S=1$ run on the joint problem 
\begin{equation}\label{eqn:robust_PCA_joint}\min_{X,S\in\R^{m\times n}}\{\f_{\mathsf{joint}}(X,S)=\frac{1}{2}\fronorm{\Ds - (X+S)}^2:\rank(X)\leq r, \; S\in\Sset \}.\end{equation} 
Under the assumption that $\Xs$ is $\mu$-incoherent and  each column of $\Ss$ is $s$-sparse, we can then prove that the joint objective function $\f_{\mathsf{joint}}$ exhibits {\em joint} restricted strong convexity and restricted smoothness over  the pairs $\big((X,S),(\Xs,\Ss)\big)$, with $2\alpha > \beta$, whenever $(X,S)\in\R^{m\times n}_{\rank(r)} \times \Sset$ {\em and} $X$ is incoherent.
This allows simple extension of (first part of)~\thmref{global_opt} to the joint minimization problem~\eqnref{robust_PCA_joint} to guarantee that $(X,S(X))$ is the unique fixed point of PGD for joint minimization, as long as $X$ is incoherent relative to sparse matrices. Since $(\Xs,\Ss)$ is trivially a fixed point, this means that $(X,S(X))=(\Xs,\Ss)$, and in particular we get $AB^\top =\Xs$. We state this result in the following lemma, whose proof is given in~\appref{proofs}.
\begin{lemma}\label{lem:robust_PCA}
	Consider the  robust PCA problem~\eqnref{robust_PCA_problem}. Suppose that the rank-$r$ matrix $\Xs$ is $\mu$-incoherent~\eqnref{mu_incoherence}, and the sparse matrix $\Ss$ is $s$-sparse in each column. 
	Let $\kappa(\Xs)=\sigma_1(\Xs)/\sigma_r(\Xs)$.
	Then there exists a constant $c_1>0$ such that the following holds: for any second-order stationary point $(A,B)$ of the factorized objective $\g$, as in equation~\eqnref{SOSP_rpca}, if $(A,B)$ satisfies the following conditions, 
	\begin{equation}\label{eqn:mu_incoherence2} A^\top A=B^\top B\footnote{Note that this condition can be easily imposed on the factors $A$ and $B$ if we add a regularizer to the factorized objective function. In particular, by~\citet[Theorem 3]{zhu2018global}, any critical point $(A,B)$ of $\g_{\textsf{reg}}$ satisfies $A^\top A=B^\top B$ and the correspondence result still holds by~\lemref{reg_obj}. See~\secref{reg_obj}.} \text{ and } \norm{A}_{2,\infty}\leq c_2\sqrt{\sigma_1(\Xs)\frac{\mu r}{m}} \text{ and } \norm{B}_{2,\infty}\leq c_2\sqrt{\sigma_1(\Xs)\frac{\mu r}{n}}, \end{equation}
	for some constant $c_2>0$ such that $4c_2\sqrt{\frac{\kappa(\Xs) \mu r s}{\min\{m,n\}}}\leq c_1$,
	then $AB^\top = \Xs$. In other words, $\Xs=\As\Bs^\top$ is the unique ``incoherent'' second-order stationary point of $\g$. 
\end{lemma}

\noindent The result of the lemma requires the stationary point $(A,B)$ to be $\mu$-incoherent~\eqnref{mu_incoherence2}. To enforce the incoherence on the factored matrices $(A,B)$,  some works are focused on putting explicit incoherence penalty/constraint at each iterate (e.g., \citet{chen2015fast,Zheng2016Convergence,ge2017no}); while other works  have proved that each update of the factorized function $\g$ stays incoherent near the true matrix $\Xs$  {\em without} explicit incoherence regularization (e.g., \citet{ma2017implicit,chen2019noisy}). In practice simple algorithms such as gradient descent are observed to work well without incoherence regularizaiton, even globally. Examining the incoherence property of any second-order stationary point of $\g$, or a fixed point of PGD in the full-dimensional space, is therefore an interesting direction which we leave for future study.

\subsection{Matrix completion}\label{sec:matrix_completion}

Next consider the matrix completion minimization problem (\citep{candes2009exact,negahban2012restricted}) where we are given an unknown low-rank matrix $\Xs\in\R^{m\times n}$ while only a subset $\Omega\subset [m]\times [n]$ of entries are observed. Here we assume each entry $(i,j)\in\Omega$ of $\Xs$ is observed independently with probability $p$. Writing $\Pr{\Omega}{X}$ to denote the matrix whose entries are set to $0$ on $\Omega^c$, i.e., $(\Pr{\Omega}{X})_{ij}=X_{ij}\cdot \ones_{(i,j)\in\Omega}$, we solve the following minimization problem
\begin{equation}\label{eqn:matrix_completion_problem} \min_{X}\left\{ \f(X)=\frac{1}{2p}\fronorm{\Pr{\Omega}{X-\Xs}}^2: \rank(X)\leq r\right\}. \end{equation}
As in the case of robust PCA, the matrix completion problem is ill-posed without any incoherence type of conditions on the true matrix---indeed, if $\Xs$ is sparse, $\Pr{\Omega}{\Xs}$ is likely to  be a zero matrix and the optimization probem~\eqnref{matrix_completion_problem} owns a trivial solution which will be far from $\Xs$. To prevent this pathological case and allow reliable estimation, we therefore focus on recovering the incoherent matrix, as defined in~\eqnref{mu_incoherence}~\citep{candes2009exact}. 

Recent results have verified that  the matrix completion objective function is locally well-behaved near the true matrix if it satisfies the  incoherence condition. Specifically, if the matrix is initialized within $\mathcal{O}(\sigma_r(\Xs))$-neighborhood of $\Xs$, then with high probability the factorized objective $\g(A,B)=\f(AB^\top)$ satisfies restricted strong convexity and restricted smoothness type of conditions on the space of factored matrices $(A,B)$~\citep{chen2015fast,Zheng2016Convergence,ge2017no,ma2017implicit} (here randomness arises from the sampling operator $\Omega$). Adapting  this result to the original function $\f$, and  combining with~\thmref{local_opt}, we can characterize the basin of attraction for matrix completion model.
\begin{lemma}\label{lem:matrix_completion_local}
	Consider the matrix completion problem~\eqnref{matrix_completion_problem}. Suppose that the rank-r matrix $\Xs$ is $\mu$-incoherent~\eqnref{mu_incoherence}. Let $\kappa(\Xs)=\sigma_1(\Xs)/\sigma_r(\Xs)$, and suppose the sampling probability $p\geq c_1\frac{\mu^2 r^2\kappa^4(\Xs)(m+n)\log(m+n)}{mn}$ for  $c_1>0$. Define the local region around $\Xs$ given by
	\begin{multline*} \mathcal{N}(\Xs) = \left\{X\in\R^{m\times n}_{\rank(r)}: \fronorm{X-\Xs}\leq 0.1\kappa^{-1}(\Xs)\sigma_r(\Xs), \text{ and } \right.\\\left.  \hspace{0in}X=AB^\top \textnormal{ where $(A,B)$ satisfies~\eqnref{mu_incoherence2}}\right\} . \nonumber
	\end{multline*}
	Then the following holds with probability larger than $1-\mathcal{O}(\min\{m,n\}^{-3})$: for any second-order stationary point $(A,B)$ of the factorized objective $\g$, if $X=AB^\top \in \mathcal{N}(\Xs)$, 
	then $X=\Xs$. In other words, $\Xs=\As\Bs^\top$ is the unique ``incoherent'' second-order stationary point of $\g$ in the region $\mathcal{N}(\Xs)$.
\end{lemma}
\noindent The proof is given in~\appref{proofs}.

The initialization condition, i.e. the condition $\fronorm{X-\Xs}\leq 0.1\kappa^{-1}(\Xs)\sigma_r(\Xs)$ in $\mathcal{N}(\Xs)$, can typically be achieved via spectral initialization, or relaxing the rank constraint to a convex constraint (such as nuclear-norm penalty) and solving the corresponding convex problem. Similarly to the case of robust PCA, the incoherence of the factored matrices $(A,B)$ can be achieved via explicit constraint/penalty, or in certain settings the incoherence is implicitly imposed via an iterative algorithm such as gradient descent on the factorized space.

\section{Discussion}
In this paper, we establish a  connection between the full-dimensional  approach and the factorized approach for solving nonconvex low-rank optimization problems. Our main result shows that any SOSP of the factorized problem must also be a fixed point of projected gradient descent algorithms on the original function, connecting naturally the optimization landscape of the unconstrained factorized approaches with the full-dimensional rank-constrained approaches. In particular, this allows us to obtain various types of established optimality results for PGD algorithms and factorized algorithms in a single framework. 
We also illustrate applications of our framework to certain low-rank estimation problems arising in matrix signal recovery, such as matrix sensing, matrix completion, and robust PCA.
Overall, our result provides a new perspective on understanding the optimization landscape of the factorized approaches.

While the present work only considers exact fixed points of PGD and exact SOSPs of the factorized problems, finding such points is practically challenging. Standard optimization techniques such as stochastic or perturbed gradient descent are known to converge to an approximate SOSP \citep{ge2015escaping,jin2017escape}. Characterizing equivalence between approximate fixed points for full-dimensional PGD versus factorized approaches is therefore of practical interest. Another interesting direction would be to establish similar results under additional constraints on the full matrix $X=AB^\top$ or on the factorized matrices $A$ and $B$.

	\subsection*{Acknowledgements}
W.H. was supported by the NSF via the TRIPODS program and by Berkeley Institute for Data Science. R.F.B.~was partially supported by the NSF via grant DMS--1654076 and by an Alfred
P.~Sloan fellowship.
	
	\bibliographystyle{plainnat}
	\bibliography{factorized}

\begin{thebibliography}{41}
\providecommand{\natexlab}[1]{#1}
\providecommand{\url}[1]{\texttt{#1}}
\expandafter\ifx\csname urlstyle\endcsname\relax
  \providecommand{\doi}[1]{doi: #1}\else
  \providecommand{\doi}{doi: \begingroup \urlstyle{rm}\Url}\fi

\bibitem[Absil et~al.(2009)Absil, Mahony, and Sepulchre]{absil2009optimization}
P-A Absil, Robert Mahony, and Rodolphe Sepulchre.
\newblock \emph{Optimization algorithms on matrix manifolds}.
\newblock Princeton University Press, 2009.

\bibitem[Agarwal et~al.(2010)Agarwal, Negahban, and
  Wainwright]{agarwal2010fast}
Alekh Agarwal, Sahand Negahban, and Martin~J. Wainwright.
\newblock Fast global convergence rates of gradient methods for
  high-dimensional statistical recovery.
\newblock In \emph{Advances in Neural Information Processing Systems}, pages
  37--45, 2010.

\bibitem[Barber and Ha(2018)]{barber2017gradient}
Rina~Foygel Barber and Wooseok Ha.
\newblock Gradient descent with non-convex constraints: local concavity
  determines convergence.
\newblock \emph{Information and Inference: A Journal of the IMA}, 2018.

\bibitem[Becker et~al.(2013)Becker, Cevher, and
  Kyrillidis]{becker2013randomized}
Stephen Becker, Volkan Cevher, and Anastasios Kyrillidis.
\newblock Randomized low-memory singular value projection.
\newblock \emph{arXiv preprint arXiv:1303.0167}, 2013.

\bibitem[Bhojanapalli et~al.(2016{\natexlab{a}})Bhojanapalli, Kyrillidis, and
  Sanghavi]{bhojanapalli2016dropping}
Srinadh Bhojanapalli, Anastasios Kyrillidis, and Sujay Sanghavi.
\newblock Dropping convexity for faster semi-definite optimization.
\newblock In \emph{Conference on Learning Theory}, pages 530--582,
  2016{\natexlab{a}}.

\bibitem[Bhojanapalli et~al.(2016{\natexlab{b}})Bhojanapalli, Neyshabur, and
  Srebro]{bhojanapalli2016global}
Srinadh Bhojanapalli, Behnam Neyshabur, and Nati Srebro.
\newblock Global optimality of local search for low rank matrix recovery.
\newblock In \emph{Advances in Neural Information Processing Systems}, pages
  3873--3881, 2016{\natexlab{b}}.

\bibitem[Bhojanapalli et~al.(2018)Bhojanapalli, Boumal, Jain, and
  Netrapalli]{bhojanapalli2018smoothed}
Srinadh Bhojanapalli, Nicolas Boumal, Prateek Jain, and Praneeth Netrapalli.
\newblock Smoothed analysis for low-rank solutions to semidefinite programs in
  quadratic penalty form.
\newblock \emph{arXiv preprint arXiv:1803.00186}, 2018.

\bibitem[Burer and Monteiro(2003)]{burer2003nonlinear}
Samuel Burer and Renato~DC Monteiro.
\newblock A nonlinear programming algorithm for solving semidefinite programs
  via low-rank factorization.
\newblock \emph{Mathematical Programming}, 95\penalty0 (2):\penalty0 329--357,
  2003.

\bibitem[Cand{\`e}s and Recht(2009)]{candes2009exact}
Emmanuel~J. Cand{\`e}s and Benjamin Recht.
\newblock Exact matrix completion via convex optimization.
\newblock \emph{Foundations of Computational mathematics}, 9\penalty0
  (6):\penalty0 717, 2009.

\bibitem[Cand{\`e}s et~al.(2011)Cand{\`e}s, Li, Ma, and
  Wright]{candes2011robust}
Emmanuel~J. Cand{\`e}s, Xiaodong Li, Yi~Ma, and John Wright.
\newblock Robust principal component analysis?
\newblock \emph{Journal of the ACM (JACM)}, 58\penalty0 (3):\penalty0 11, 2011.

\bibitem[Candes et~al.(2015)Candes, Li, and Soltanolkotabi]{candes2015phase}
Emmanuel~J. Candes, Xiaodong Li, and Mahdi Soltanolkotabi.
\newblock Phase retrieval via {W}irtinger flow: Theory and algorithms.
\newblock \emph{IEEE Transactions on Information Theory}, 61\penalty0
  (4):\penalty0 1985--2007, 2015.

\bibitem[Chandrasekaran et~al.(2011)Chandrasekaran, Sanghavi, Parrilo, and
  Willsky]{chandrasekaran2011rank}
Venkat Chandrasekaran, Sujay Sanghavi, Pablo~A. Parrilo, and Alan~S. Willsky.
\newblock Rank-sparsity incoherence for matrix decomposition.
\newblock \emph{SIAM Journal on Optimization}, 21\penalty0 (2):\penalty0
  572--596, 2011.

\bibitem[Chen and Li(2017)]{chen2017memory}
Ji~Chen and Xiaodong Li.
\newblock Memory-efficient kernel {PCA} via partial matrix sampling and
  nonconvex optimization: a model-free analysis of local minima.
\newblock \emph{arXiv preprint arXiv:1711.01742}, 2017.

\bibitem[Chen and Wainwright(2015)]{chen2015fast}
Yudong Chen and Martin~J. Wainwright.
\newblock Fast low-rank estimation by projected gradient descent: General
  statistical and algorithmic guarantees.
\newblock \emph{arXiv preprint arXiv:1509.03025}, 2015.

\bibitem[Chen et~al.(2019)Chen, Chi, Fan, Ma, and Yan]{chen2019noisy}
Yuxin Chen, Yuejie Chi, Jianqing Fan, Cong Ma, and Yuling Yan.
\newblock Noisy matrix completion: Understanding statistical guarantees for
  convex relaxation via nonconvex optimization.
\newblock \emph{arXiv preprint arXiv:1902.07698}, 2019.

\bibitem[Ge et~al.(2015)Ge, Huang, Jin, and Yuan]{ge2015escaping}
Rong Ge, Furong Huang, Chi Jin, and Yang Yuan.
\newblock Escaping from saddle points—online stochastic gradient for tensor
  decomposition.
\newblock In \emph{Conference on Learning Theory}, pages 797--842, 2015.

\bibitem[Ge et~al.(2016)Ge, Lee, and Ma]{ge2016matrix}
Rong Ge, Jason~D. Lee, and Tengyu Ma.
\newblock Matrix completion has no spurious local minimum.
\newblock In \emph{Advances in Neural Information Processing Systems}, pages
  2973--2981, 2016.

\bibitem[Ge et~al.(2017)Ge, Jin, and Zheng]{ge2017no}
Rong Ge, Chi Jin, and Yi~Zheng.
\newblock No spurious local minima in nonconvex low rank problems: A unified
  geometric analysis.
\newblock \emph{arXiv preprint arXiv:1704.00708}, 2017.

\bibitem[Jain et~al.(2013)Jain, Netrapalli, and Sanghavi]{jain2013low}
Prateek Jain, Praneeth Netrapalli, and Sujay Sanghavi.
\newblock Low-rank matrix completion using alternating minimization.
\newblock In \emph{Proceedings of the forty-fifth annual ACM symposium on
  Theory of computing}, pages 665--674. ACM, 2013.

\bibitem[Jain et~al.(2014)Jain, Tewari, and Kar]{jain2014iterative}
Prateek Jain, Ambuj Tewari, and Purushottam Kar.
\newblock On iterative hard thresholding methods for high-dimensional
  {M}-estimation.
\newblock In \emph{Advances in Neural Information Processing Systems}, pages
  685--693, 2014.

\bibitem[Jin et~al.(2017)Jin, Ge, Netrapalli, Kakade, and
  Jordan]{jin2017escape}
Chi Jin, Rong Ge, Praneeth Netrapalli, Sham~M. Kakade, and Michael~I. Jordan.
\newblock How to escape saddle points efficiently.
\newblock \emph{arXiv preprint arXiv:1703.00887}, 2017.

\bibitem[Keshavan et~al.(2010)Keshavan, Montanari, and Oh]{keshavan2010matrix}
Raghunandan~H. Keshavan, Andrea Montanari, and Sewoong Oh.
\newblock Matrix completion from a few entries.
\newblock \emph{IEEE transactions on information theory}, 56\penalty0
  (6):\penalty0 2980--2998, 2010.

\bibitem[Liu and Barber(2018)]{liu2018between}
Haoyang Liu and Rina~Foygel Barber.
\newblock Between hard and soft thresholding: optimal iterative thresholding
  algorithms.
\newblock \emph{arXiv preprint arXiv:1804.08841}, 2018.

\bibitem[Ma et~al.(2017)Ma, Wang, Chi, and Chen]{ma2017implicit}
Cong Ma, Kaizheng Wang, Yuejie Chi, and Yuxin Chen.
\newblock Implicit regularization in nonconvex statistical estimation: Gradient
  descent converges linearly for phase retrieval, matrix completion and blind
  deconvolution.
\newblock \emph{arXiv preprint arXiv:1711.10467}, 2017.

\bibitem[Mishra et~al.(2013)Mishra, Meyer, Bach, and Sepulchre]{mishra2013low}
Bamdev Mishra, Gilles Meyer, Francis Bach, and Rodolphe Sepulchre.
\newblock Low-rank optimization with trace norm penalty.
\newblock \emph{SIAM Journal on Optimization}, 23\penalty0 (4):\penalty0
  2124--2149, 2013.

\bibitem[Negahban and Wainwright(2012)]{negahban2012restricted}
Sahand Negahban and Martin~J. Wainwright.
\newblock Restricted strong convexity and weighted matrix completion: Optimal
  bounds with noise.
\newblock \emph{Journal of Machine Learning Research}, 13\penalty0
  (May):\penalty0 1665--1697, 2012.

\bibitem[Negahban et~al.(2012)Negahban, Ravikumar, Wainwright, and
  Yu]{negahban2012unified}
Sahand Negahban, Pradeep Ravikumar, Martin~J. Wainwright, and Bin Yu.
\newblock A unified framework for high-dimensional analysis of {M}-estimators
  with decomposable regularizers.
\newblock \emph{Statistical Science}, 27\penalty0 (4):\penalty0 538--557, 2012.

\bibitem[Oymak et~al.(2018)Oymak, Recht, and Soltanolkotabi]{oymak2018sharp}
Samet Oymak, Benjamin Recht, and Mahdi Soltanolkotabi.
\newblock Sharp time--data tradeoffs for linear inverse problems.
\newblock \emph{IEEE Transactions on Information Theory}, 64\penalty0
  (6):\penalty0 4129--4158, 2018.

\bibitem[Park et~al.(2016)Park, Kyrillidis, Caramanis, and
  Sanghavi]{park2016non}
Dohyung Park, Anastasios Kyrillidis, Constantine Caramanis, and Sujay Sanghavi.
\newblock Non-square matrix sensing without spurious local minima via the
  {B}urer-{M}onteiro approach.
\newblock \emph{arXiv preprint arXiv:1609.03240}, 2016.

\bibitem[Recht et~al.(2010)Recht, Fazel, and Parrilo]{recht2010guaranteed}
Benjamin Recht, Maryam Fazel, and Pablo~A Parrilo.
\newblock Guaranteed minimum-rank solutions of linear matrix equations via
  nuclear norm minimization.
\newblock \emph{SIAM review}, 52\penalty0 (3):\penalty0 471--501, 2010.

\bibitem[Rockafellar and Wets(2009)]{rockafellar2009variational}
R.~Tyrrell Rockafellar and Roger J.-B. Wets.
\newblock \emph{Variational analysis}, volume 317.
\newblock Springer Science \& Business Media, 2009.

\bibitem[Shen and Li(2017)]{shen2017tight}
Jie Shen and Ping Li.
\newblock A tight bound of hard thresholding.
\newblock \emph{The Journal of Machine Learning Research}, 18\penalty0
  (1):\penalty0 7650--7691, 2017.

\bibitem[Soltani and Hegde(2017)]{soltani2017fast}
Mohammadreza Soltani and Chinmay Hegde.
\newblock Fast low-rank matrix estimation without the condition number.
\newblock \emph{arXiv preprint arXiv:1712.03281}, 2017.

\bibitem[Tu et~al.(2015)Tu, Boczar, Simchowitz, Soltanolkotabi, and
  Recht]{tu2015low}
Stephen Tu, Ross Boczar, Max Simchowitz, Mahdi Soltanolkotabi, and Benjamin
  Recht.
\newblock Low-rank solutions of linear matrix equations via {P}rocrustes flow.
\newblock \emph{arXiv preprint arXiv:1507.03566}, 2015.

\bibitem[Vandereycken(2013)]{vandereycken2013low}
Bart Vandereycken.
\newblock Low-rank matrix completion by {R}iemannian optimization.
\newblock \emph{SIAM Journal on Optimization}, 23\penalty0 (2):\penalty0
  1214--1236, 2013.

\bibitem[Yuan et~al.(2018)Yuan, Li, and Zhang]{yuan2018gradient}
Xiao-Tong Yuan, Ping Li, and Tong Zhang.
\newblock Gradient hard thresholding pursuit.
\newblock \emph{Journal of Machine Learning Research}, 18\penalty0
  (166):\penalty0 1--43, 2018.

\bibitem[Zhang et~al.(2018)Zhang, Josz, Sojoudi, and Lavaei]{zhang2018much}
Richard Zhang, C{\'e}dric Josz, Somayeh Sojoudi, and Javad Lavaei.
\newblock How much restricted isometry is needed in nonconvex matrix recovery?
\newblock In \emph{Advances in neural information processing systems}, pages
  5586--5597, 2018.

\bibitem[Zhang et~al.(2019)Zhang, Sojoudi, and Lavaei]{zhang2019sharp}
Richard Zhang, Somayeh Sojoudi, and Javad Lavaei.
\newblock Sharp restricted isometry bounds for the inexistence of spurious
  local minima in nonconvex matrix recovery.
\newblock \emph{arXiv preprint arXiv:1901.01631}, 2019.

\bibitem[Zheng and Lafferty(2015)]{zheng2015convergent}
Qinqing Zheng and John Lafferty.
\newblock A convergent gradient descent algorithm for rank minimization and
  semidefinite programming from random linear measurements.
\newblock In \emph{Advances in Neural Information Processing Systems}, pages
  109--117, 2015.

\bibitem[Zheng and Lafferty(2016)]{Zheng2016Convergence}
Qinqing Zheng and John Lafferty.
\newblock Convergence analysis for rectangular matrix completion using
  {B}urer-{M}onteiro factorization and gradient descent.
\newblock \emph{arXiv preprint arXiv:1605.07051}, 2016.

\bibitem[Zhu et~al.(2018)Zhu, Li, Tang, and Wakin]{zhu2018global}
Zhihui Zhu, Qiuwei Li, Gongguo Tang, and Michael~B. Wakin.
\newblock Global optimality in low-rank matrix optimization.
\newblock \emph{IEEE Transactions on Signal Processing}, 66\penalty0
  (13):\penalty0 3614--3628, 2018.

\end{thebibliography}

\appendix

\section{Additional proofs}\label{app:proofs}

\begin{proof}[Proof of \thmref{restricted_opt_lowerbd}]
	
	Without loss of generality, take $m\leq n$. Define the matrices
	\[X_0 = \sum_{i=1}^r \ee{i}\ee{i}^\top\textnormal{ and }
	X_1 =  \sum_{i=r+1}^{r+r'} \ee{i}\ee{i}^\top,\]
	and
	\[M =\left(\begin{array}{cc}\mathbf{0}_{r\times r} & \mathbf{1}_{r\times (n-r)} \\ \mathbf{1}_{(m-r)\times r}&\mathbf{0}_{(m-r)\times(n-r)}\end{array}\right).\] 
	Writing $\circ$ to denote the elementwise product, we will consider the objective function
	\[\f(X) = - \beta\cdot  \inner{X_1}{X-X_0} + \frac{\alpha}{2} \cdot  \fronorm{X-X_0}^2 + \frac{\beta-\alpha}{2}\cdot \fronorm{M\circ (X-X_0)}^2,\]
	which clearly is $\alpha$-strongly convex and $\beta$-smooth (and therefore trivially satisfies $\alpha$-RSC and $\beta$-RSM).
	Define
	\[A_0 = \left(\begin{array}{c}\ident{r} \\ \mathbf{0}_{(m-r)\times r}\end{array}\right)\textnormal{ and }B_0 = \left(\begin{array}{c}\ident{r} \\ \mathbf{0}_{(m-r)\times r}\end{array}\right).\]
	Then $A_0B_0^\top = X_0$, and a trivial calculation verifies that
	\[\f(A_0B_0^\top) = \f(X_0) = 0 > \f\big(\kappa \cdot X_1\big) \geq \min_{\rank(Y)\leq r'}\f(Y).\]
	Therefore, $A_0B_0^\top$ does not satisfy restricted optimality relative to the rank $r'$. 
	
	Now it remains to be shown that the pair $(A_0,B_0)$ is a second-order stationary point of the factorized objective function $\g(A,B)$.
	We can trivially see that $\nabla\f(X_0)=\beta X_1$, and so $\nabla \f(X_0)^\top A_0 =0$
	and $\nabla\f(X_0) B_0=0$, verifying that $(A_0,B_0)$ satisfies the first-order conditions.
	Now we examine the second-order conditions. We need to prove that, for any pair of matrices $(A_1,B_1)$,
	the operator $\nabla^2\g(A_0,B_0)$ maps $(A_1,B_1)\times(A_1,B_1)$
	to a nonnegative value. Using our earlier calculation~\eqnref{second_deriv_g_psd} to derive $\nabla^2\g(A,B)$, 
	we can calculate 
	\begin{align}
		\notag&\nabla^2\g(A_0,B_0)\Big((A_1,B_1),(A_1,B_1)\Big)\\
		\notag&=2\inner{\nabla \f(X_0)}{A_1B_1^\top}+ \nabla^2\f(X_0)\Big(A_0B_1^\top + A_1B_0^\top, A_0B_1^\top + A_1B_0^\top\Big)\\
		\label{eqn:check_psd_step1}
		&=2\beta\cdot \inner{X_1}{A_1B_1^\top}+\alpha\cdot \fronorm{A_0B_1^\top + A_1B_0^\top}^2 + (\beta-\alpha)\cdot\fronorm{M\circ( A_0B_1^\top + A_1B_0^\top)}^2,
	\end{align}
	where the last step holds by definition of $\f$.
	Now we split the matrices $A_1$ and $B_1$ into block form, writing
	\[A_1 =\left(\begin{array}{c}A_1' \\ A_1''\end{array}\right), \quad B_1 =\left(\begin{array}{c}B_1' \\ B_1''\end{array}\right),\]
	where $A_1'$ and $B_1'$ contain the first $r$ rows of $A_1$ and of $B_1$, respectively.
	Then, plugging in the definitions of $X_1$, $A_0$, $B_0$, and $M$, the expression in~\eqnref{check_psd_step1} can be rewritten as
	\[2\beta\cdot \textnormal{trace}\left(A_1''B_1''{}^\top\right)+\alpha\cdot \Fronorm{\left(\begin{array}{cc}A_1' + B_1'{}^\top & B_1''{}^\top \\ A_1'' & \mathbf{0}_{(m-r)\times(n-r)}\end{array}\right)}^2 + (\beta-\alpha)\cdot\Fronorm{\left(\begin{array}{cc}\mathbf{0}_{r\times r} & B_1''{}^\top \\ A_1'' & \mathbf{0}_{(m-r)\times(n-r)}\end{array}\right)}^2.\]
	This is trivially lower-bounded by
	\[2\beta\cdot \textnormal{trace}\left(A_1''B_1''{}^\top\right)+\beta\cdot\fronorm{A_1''}^2 + \beta\cdot\fronorm{B_1''}^2.\]
	Using the fact that $\big|\textnormal{trace}(YZ) \big|\leq \fronorm{Y}\fronorm{Z}$ for all matrices $Y,Z$, this expression is clearly nonnegative. We have therefore proved that
	$\nabla^2\g(A_0,B_0)\succeq 0$, thus verifying that $(A_0,B_0)$ is a SOSP and proving the desired result.
\end{proof}

\begin{proof}[Proof of~\lemref{localmin_vs_fixedpt}]
	First, we must have $\nabla\f(X)^\top U_X =0$ and $\nabla \f(X) V_X = 0$ since $X$ is a critical point of the rank-constrained
	minimization problem. Next, let $X=U_X\cdot \textnormal{diag}\{\sigma_1,\dots,\sigma_r\}\cdot V_X^\top$  be a SVD of $X$ and $u_\star\in\R^m$ and $v_\star\in\R^n$ be the top singular vectors of the gradient $\nabla\f(X)$.
	For $t\in [0,1]$, define 
	\[X_t = \sum_{i=1}^{r-1}\sigma_i u_{X,i}v_{X,i}^\top + \sigma_r \left[\sqrt{1-t^2} \cdot u_{X,r} + t u_\star \right]\left[\sqrt{1-t^2} \cdot v_{X,r} - t v_\star \right]^\top . \]
	Since $\nabla\f(X)^\top U_X = 0$ and $\nabla \f(X) V_X = 0$, some calculations yield
	\begin{equation*}
		\inner{\nabla\f(X)}{X_t - X } =\inner{\nabla \f(X)}{-t^2\sigma_r u_\star v_\star} =  -t^2\sigma_r \norm{\nabla\f(X)} = -\frac{1}{2\sigma_r}\norm{\nabla\f(X)}\fronorm{X_t-X}^2.
	\end{equation*}
	Now, if $\norm{\nabla\f(X)}>\beta_{\mathsf{local}}\cdot \sigma_r$, then we can find some small $\delta>0$ such that
	\begin{equation*}
		\inner{\nabla\f(X)}{X_t - X } < -\frac{\beta_{\mathsf{local}}+\delta}{2}\fronorm{X_t-X}^2
	\end{equation*}
	for all $t\in(0,1]$
	(note that this step uses the fact that $\fronorm{X_t - X}>0$ for all $t\neq 0$).
	
	On the other hand, by definition of $\beta_{\mathsf{local}}$~\eqnref{local_smth}, for sufficiently small $t_0> 0$ we have
	\[\f(X_t) \leq \f(X)+\inner{\nabla\f(X)}{X_t-X} + \frac{\beta_{\mathsf{local}}(X) + \delta}{2}\fronorm{X_t-X}^2 \]
	for all $0\leq t\leq t_0$.
	Combining these calculations, for all $t\in(0,t_0]$ we have
	\[\f(X_t)   < \f(X),\]
	which contradicts the assumption that $X$ is a local minimum.
\end{proof}

\begin{proof}[Proof of~\lemref{reg_obj}]
	In order for the proof of \thmref{main} to hold for this new setting, we need to verify that the 
	equations~\eqnref{first_deriv_g_zero} and~\eqnref{second_deriv_g_psd} both
	hold.
	
	By~\citet[Theorem 3, (16)--(18), and Remark 8]{zhu2018global},  for any pair $(A,B)$ for which $A^\top A = B^\top B$,
	the first derivative satisfies
	\begin{equation}\label{eqn:g_reg_1st_deriv}\nabla\g_{\mathsf{reg}}(A,B) = \nabla\g(A,B)\end{equation}
	while the second derivative $\nabla^2\g_{\mathsf{reg}}(A,B)$
	maps $(A_1,B_1)\times (A_1,B_1)$ to
	\begin{equation}\label{eqn:g_reg_2nd_deriv}\nabla^2\g(A,B)\Big((A_1,B_1),(A_1,B_1)\Big) + 4\lambda\fronorm{A^\top A_1 + A_1^\top A - B^\top B_1 - B_1^\top B}^2,\end{equation}
	and furthermore $A^\top A = B^\top B$ holds for any critical point $(A,B)$ of $\g_{\mathsf{reg}}$.
	Comparing to the proof of \thmref{main}, the first-derivative property therefore verify that~\eqnref{first_deriv_g_zero} holds,
	while the second-derivative property verifies that~\eqnref{second_deriv_g_psd}
	holds for any $(A_1,B_1)$ with $A^\top A_1 = 0$ and $B^\top B_1=0$.
	Now, following the proof of 
	\thmref{main}, for both the full-rank case and the rank-deficient case, we set $(A_1,B_1) = (-u_\star z^\top , 
	v_\star z'{}^\top )$ for some vectors $z,z'$, where $u_\star,v_\star$ are the top singular vectors of $\nabla\f(AB^\top)$ and, therefore, satisfy $u_\star \perp A$ and $v_\star \perp B$ 
	by~\eqnref{first_deriv_g_zero}. This means that we indeed have  $A^\top A_1 = 0$ and $B^\top B_1=0$,
	and so ~\eqnref{second_deriv_g_psd}
	holds for the relevant choice of $(A_1,B_1)$. This is sufficient for the proof of \thmref{main}(a) to yield the desired result for the regularized setting.
	To verify that \thmref{main}(b) holds in this setting, consider any $X$ that is a local minimum of $\min_{\rank(X)\leq r}\f(X)$.
	Define $A = U_X\cdot  \textnormal{diag}\{\sigma_1,\dots,\sigma_r\}^{1/2}$ and $B = V_X \cdot  \textnormal{diag}\{\sigma_1,\dots,\sigma_r\}^{1/2}$.
	Then clearly $A^\top A = B^\top B$. Since \thmref{main}(a) implies that $(A,B)$ is a SOSP of $\g$,
	we know that $\nabla\g(A,B)=0$ and $\nabla^2\g(A,B)\succeq 0$. Combined with~\eqnref{g_reg_1st_deriv} and~\eqnref{g_reg_2nd_deriv},
	this proves that $(A,B)$ is a SOSP of $\g_{\mathsf{reg}}$.

\end{proof}

\begin{proof}[Proof of \lemref{RSM_Hessian}]
	Define $Y_t = (A + t A_1)(B+ tB_1)^\top$ for $t>0$. Note that $\fronorm{X-Y_t}\rightarrow 0$ as $t\rightarrow 0$. By definition of $\beta_{\mathsf{local}}(X)$,
	\[\lim\sup_{t\rightarrow 0} \frac{\f(Y_t) - \f(X) - \inner{\nabla\f(X)}{Y_t - X}}{\frac{1}{2}\fronorm{X-Y_t}^2} \leq\beta_{\mathsf{local}}(X).\]
	Since $\f$ is twice-differentiable at $X$, we can also take a Taylor expansion to see that
	\[\lim\inf_{t\rightarrow 0} \frac{\f(Y_t) - \f(X) - \inner{\nabla\f(X)}{Y_t - X} - \frac{1}{2}\nabla^2\f(X)\big(Y_t - X,Y_t-X)}{\frac{1}{2}\fronorm{X-Y_t}^2} =  0.\]
	Combining these two, we see that
	\[\lim\sup_{t\rightarrow 0} \frac{\nabla^2\f(X)\big(Y_t - X,Y_t-X)}{\fronorm{X-Y_t}^2} \leq \beta_{\mathsf{local}}(X).\]
	Now we calculate this fraction. Since $Y_t - X = t\cdot (AB_1^\top + A_1B^\top) + t^2\cdot A_1B_1^\top$, we have
	\[\fronorm{X-Y_t}^2 = t^2 \fronorm{AB_1^\top + A_1B^\top}^2 + \mathcal{O}(t^3)\]
	and
	\[\nabla^2\f(X)\big(Y_t - X,Y_t-X) = t^2\nabla^2\f(X)\big(AB_1^\top + A_1B^\top,AB_1^\top + A_1B^\top\big) + \mathcal{O}(t^3),\]
	and therefore,
	\[\lim\sup_{t\rightarrow 0}\frac{\nabla^2\f(X)\big(Y_t - X,Y_t-X)}{\fronorm{X-Y_t}^2}  = \frac{\nabla^2\f(X)\big(AB_1^\top + A_1B^\top,AB_1^\top + A_1B^\top\big) }{ \fronorm{AB_1^\top + A_1B^\top}^2 },\]
	(as long as we are not in the degenerate case that $\fronorm{AB_1^\top + A_1B^\top}=0$---but if this were the case, then the result would hold trivially).
	Combining everything, we have proved the desired bound.\end{proof}

\begin{proof}[Proof of~\lemref{PGD_alphabeta}]
	By assumption, $X_0$ is a fixed point of PGD for some step size $\eta_0>0$, meaning that
	\[X_0 =  \pr{r}\big(X_0 - \eta_0 \nabla \f(X_0)\big). \]
	For the case  $\rank(X_0)=r$,
	then $X_0$ is a solution to the quadratic problem with rank constraint (by definition of projection), i.e.
	\[X_0 = \argmin_{\rank(X)\leq r} \fronorm{X_0-\eta_0\nabla\f(X_0) - X}^2, \]
	Then, in the case that $\rank(X_0)=r$, \citep[Lemma 7]{barber2017gradient} proves a first-order optimality condition for rank-constrained optimization:
	\[\inner{X_1 - X_0}{\nabla\f(X_0)} \geq -\frac{1}{2\sigma_r(X_0)}\norm{\nabla\f(X_0)}\fronorm{X_0 - X_1}^2.\]
	Combined with the $\alpha$-RSC assumption over $\mathcal{X}$, we see that
	\begin{multline}\label{eqn:fX_fXh_1}\f(X_1)\geq \f(X_0) + \inner{X_1 - X_0}{\nabla\f(X_0)} + \frac{\alpha}{2}\fronorm{X_0-X_1}^2 \\\geq \f(X_0) + \frac{1}{2}\left(\alpha - \frac{\norm{\nabla\f(X_0)}}{\sigma_r(X_0)}\right)\fronorm{X_0-X_1}^2.\end{multline}
	If instead $\rank(X_0)<r$ then $\nabla\f(X_0)=0$ by the conditions of a fixed point~\eqnref{PGD_SP}, and so
	\begin{equation}\label{eqn:fX_fXh_1_alt}\f(X_1)\geq \f(X_0) + \inner{X_1 - X_0}{\nabla\f(X_0)} + \frac{\alpha}{2}\fronorm{X_0-X_1}^2 \\= \f(X_0) + \frac{1}{2}\alpha\fronorm{X_0-X_1}^2.\end{equation}
	Applying the same arguments with the roles of $X_0$ and $X_1$ reversed yields
	\begin{equation}\label{eqn:fX_fXh_2}\f(X_0)\geq  \f(X_1) +  \frac{1}{2}\left(\alpha -\frac{\norm{\nabla\f(X_1)}}{\sigma_r(X_1)}\right)\fronorm{X_0-X_1}^2\end{equation}
	if $\rank(X_1)=r$, or
	\begin{equation}\label{eqn:fX_fXh_2_alt}\f(X_0)\geq  \f(X_1) +  \frac{1}{2}\alpha\fronorm{X_0-X_1}^2\end{equation}
	if $\rank(X_1)<r$.
	
	Now suppose $\rank(X_0)=\rank(X_1)=r$. Adding the two inequalities~\eqnref{fX_fXh_1} and~\eqnref{fX_fXh_2} yields
	\[0 \geq \frac{1}{2} \left(2\alpha  -  \frac{\norm{\nabla\f(X_0)}}{\sigma_r(X_0)}- \frac{\norm{\nabla\f(X_1)}}{\sigma_r(X_1)}\right)\fronorm{X_0-X_1}^2.\]
	This implies that either $X_0 = X_1$, or 
	\[\frac{\norm{\nabla\f(X_0)}}{\sigma_r(X_0)}+  \frac{\norm{\nabla\f(X_1)}}{\sigma_r(X_1)}\geq 2\alpha,\]
	as desired. If instead $\rank(X_0)=r$ and $\rank(X_1)<r$, then adding~\eqnref{fX_fXh_1} and~\eqnref{fX_fXh_2_alt} yields
	\[0 \geq \frac{1}{2} \left(2\alpha  -  \frac{\norm{\nabla\f(X_0)}}{\sigma_r(X_0)}\right)\fronorm{X_0-X_1}^2\]
	and therefore since $X_0\neq X_1$ we have
	\[\frac{\norm{\nabla\f(X_0)}}{\sigma_r(X_0)}\geq 2\alpha.\]
	If instead $\rank(X_0)<r$ and $\rank(X_1)=r$, this case is symmetric to the one above. Finally if $\rank(X_0)<r$ and $\rank(X_1)<r$,
	then adding~\eqnref{fX_fXh_1_alt} and~\eqnref{fX_fXh_2_alt} proves that we must have $X_0=X_1$ since $\alpha>0$.
\end{proof}

\begin{proof}[Proof of \lemref{globalmin_is_stationary}]
	This lemma is an easy consequence of~\lemref{localmin_vs_fixedpt}. First, since $\Xh$ is a global minimum, it is also a local minimum and so $\Xh$ satisfies the conditions~\eqnref{PGD_SP} with $\eta=1/\beta_{\textsf{local}}(\Xh)$. Since the set $\mathcal{X}$ is open relative to the set of low-rank matrices, it contains an intersection of a neighborhood of $\Xh$ and the set of low-rank matrices. Then comparing the definition of $\beta$-RSM with that of $\beta_{\textsf{local}}(\Xh)$~\eqnref{local_smth}, it follows that $\beta_{\textsf{local}}(\Xh)\leq \beta$, and in particular, $\Xh$ also satisfies  the conditions~\eqnref{PGD_SP} with $\eta=1/\beta$. This proves that $\Xh$ is a fixed point of PGD at step size $\eta\leq 1/\beta$.
\end{proof}

\begin{proof}[Proof of~\lemref{robust_PCA}]
	First we prove that the joint objective function $\f_{\mathsf{joint}}(X,S)$, defined in equation~\eqnref{robust_PCA_joint}, satisfies joint $\alpha$-RSC/$\beta$-RSM, that is 
	\[ \frac{1}{2}\fronorm{\Ds - (X+S)}^2 \geq \frac{\alpha}{2}\fronorm{X-\Xs}^2 + \frac{\alpha}{2} \fronorm{S-\Ss}^2,\]
	and similarly for $\beta$-RSM,
	over the set 
	\begin{equation}\label{eqn:robust_PCA_set}\{(X,S)\in\R^{m\times n}_{\rank(r)}\times \Sset: \textnormal{ $X=AB^\top$ where $(A,B)$ satisfies}~\eqnref{mu_incoherence2}\}.\end{equation} Given the data matrix $\Ds=\Xs+\Ss$, we have
	\begin{equation}\label{eqn:rpca1}  \frac{1}{2}\fronorm{\Ds - (X+S)}^2  = \frac{1}{2}\fronorm{X-\Xs}^2 + \frac{1}{2}\fronorm{S-\Ss}^2 +2\inner{X-\Xs}{S-\Ss}. \end{equation}
	To bound the term $\inner{X-\Xs}{S-\Ss}$, we closely follow~\citet[Corollary 6]{chen2015fast} and extend their result to the asymmetric and global case (their work assumes $X=AA^\top$ and verifies RSC locally on the factorized space). 
	Note that since $X=AB^\top$ for some $(A,B)$ satisfying $A^\top A = B^\top B$~\eqnref{mu_incoherence2}, it follows that $A=U\sqrt{\Sigma}R$ and $B=V\sqrt{\Sigma}R$ for some $R$, where $X=U\Sigma V^\top$ is a SVD of $X$ and $R$ is a rotation matrix, i.e. $R^\top R=\ident{r}$.
	Without loss of generality, we assume $R$ is chosen to be the best transformation to $(\As,\Bs)$, i.e. 
	\[ R\in\argmin_{\widetilde{R}\in\R^{r\times r}}\{\fronorm{(\widetilde{A}^\top,\widetilde{B}^\top)^\top \widetilde{R}- (\As^\top,\Bs^\top)^\top}:\widetilde{A}=U\sqrt{\Sigma},\widetilde{B}=V\sqrt{\Sigma}, \widetilde{R}^\top \widetilde{R} = \ident{r} \}.\]
	Writing $X-\Xs = A(B-\Bs)^\top + (A-\As)\Bs^\top$, then:
	\begin{align*}
		&| \inner{X-\Xs}{S-\Ss}| = |\inner{B^\top-\Bs^\top}{A^\top (S-\Ss)}| + |\inner{A^\top-\As^\top}{\Bs^\top(S-\Ss)}|\\
		&\leq \fronorm{B-\Bs}\fronorm{A^\top (S-\Ss)} + \fronorm{A-\As}\fronorm{\Bs^\top(S-\Ss)} \\
		&\leq  \fronorm{B-\Bs}\sqrt{\sum_{j=1}^n \norm{A^\top (S-\Ss)e_j}_2^2} + \fronorm{A-\As}\sqrt{\sum_{j=1}^n \norm{\Bs^\top (S-\Ss)e_j}_2^2}\\
		&\leq \fronorm{B-\Bs}\sqrt{\sum_{j=1}^n \norm{A}_{2,\infty}^2\norm{(S-\Ss)e_j}_1^2} + \fronorm{A-\As}\sqrt{\sum_{j=1}^n \norm{\Bs}^2_{2,\infty} \norm{(S-\Ss)e_j}_1^2},
	\end{align*}
	where $e_j$ denotes the $j$-th standard basis vector. Since $X$ and $\Xs$ are both $\mu$-incoherent~\eqnref{mu_incoherence2}, we further have $\norm{A}_{2,\infty}\leq c_2\sqrt{\frac{\sigma_1(\Xs)\mu r}{m}}$ and $\norm{\Bs}_{2,\infty} \leq  c_2\sqrt{\frac{\sigma_1(\Xs)\mu r}{n}}$, while for each column of $S$, we have $\norm{Se_j}_0\leq \norm{\Ss e_j}_0=s$ and thus $\norm{(S-\Ss)e_j}_1\leq \sqrt{2s}\norm{(S-\Ss)e_j}_2$. Putting  these bounds together, we have
	\[ | \inner{X-\Xs}{S-\Ss}|  \leq c_2\sqrt{\frac{2\sigma_1(\Xs)\mu r s}{\min\{m,n\}}}\fronorm{S-\Ss}(\fronorm{A-\As} + \fronorm{B-\Bs}). \]
	From~\citet[Lemma 5.4]{tu2015low} and~\citet[Lemma 4]{Zheng2016Convergence}, we know that $\sqrt{\sigma_r(\Xs)}(\fronorm{A-\As} + \fronorm{B-\Bs})\leq 2\fronorm{X-\Xs}$. Plugging into the inequality above, 
	\begin{align*}
		| \inner{X-\Xs}{S-\Ss}| &\leq 2c_2\sqrt{\frac{2\kappa(\Xs) \mu r s}{\min\{m,n\}}}\fronorm{S-\Ss}\fronorm{X-\Xs} \\
		&\leq \frac{c_1}{2}\fronorm{X-\Xs}^2 + \frac{c_1}{2}\fronorm{S-\Ss}^2,
	\end{align*} 
	where the second step uses the assumption $2c_2\sqrt{\frac{2\kappa(\Xs) \mu rs}{\min\{m,n\}}}\leq c_1$, together with the identity $ab\leq \frac{a^2+b^2}{2}$. Combining with~\eqnref{rpca1}, we have proved the joint restricted strong convexity and restricted smoothness conditions over the set~\eqnref{robust_PCA_set}, with $\alpha=1-c_1, \beta=1+c_1$.
	
	Next we give a brief outline on extending the result of first part of~\thmref{global_opt} to ensure the uniqueness of the fixed point $(X,S(X))$ of PGD on the joint problem~\eqnref{robust_PCA_joint}. Note that, by definition of the fixed point, we have (with step sizes $\eta=1$)
	\[\begin{cases}X =\Proj_r \big(X-\nabla_X\f_{\mathsf{joint}}(X,S(X)) \big),\\ S(X) = \Proj_{\Sset}\big(S(X)-\nabla_S\f_{\mathsf{joint}}(X,S(X)) \big) .\end{cases} \]
	Since $\norm{\Ss{}_j}_0\leq s$ for each column of $\Ss$, by definition of projection operator this implies that
	\[ \fronorm{S(X)_j -(S(X)_j-\nabla_{S_j}\f_{\mathsf{joint}}(X,S(X))) }^2\leq \fronorm{\Ss{}_j-(S(X)_j-\nabla_{S_j}\f_{\mathsf{joint}}(X,S(X)))}^2. \]
	Rearranging terms, and combining across columns $j=1,\ldots,n$, we obtain
	\[\inner{\Ss-S(X)}{\nabla_S\f_{\mathsf{joint}}(X,S(X)) }\geq -\frac{1}{2}\fronorm{\Ss-S(X)}^2. \]
	Turning to $X$, in the case that $\rank(X)=r$, by~\citet[Lemma 7]{barber2017gradient},  we know that 
	\[\inner{\Xs-X}{\nabla_X\f_{\mathsf{joint}}(X,S(X))}\geq -\frac{1}{2\sigma_r(X)}\norm{\nabla_X\f_{\mathsf{joint}}(X,S(X))}\fronorm{\Xs-X}^2. \]
	Instead, if $\rank(X)<r$, by condition~\eqnref{PGD_SP}, $\nabla_X\f_{\mathsf{joint}}(X,S(X))=0$. Following the same arguement as in the proof of~\lemref{PGD_alphabeta}, then  (with $\alpha=1-c_1$ and $\nabla\f_{\mathsf{joint}}(\Xs,\Ss)=0$),
	\begin{equation}\label{eqn:joint_full_rank}0\geq \frac{1}{2}\left(2(1-c_1)-\frac{\norm{\nabla_X\f_{\mathsf{joint}}(X,S(X))}}{\sigma_r(X)}\right)\fronorm{X-\Xs}^2 +  \left(\frac{1-2c_1}{2} \right)\fronorm{S(X)-\Ss}^2,\end{equation}
	if $\rank(X)=r$, or
	\begin{equation}\label{eqn:joint_rank_deficient}0\geq (1-c_1)\fronorm{X-\Xs}^2 +  \left(\frac{1-2c_1}{2} \right)\fronorm{S(X)-\Ss}^2,\end{equation}
	if $\rank(X)<r$.
	
	Now suppose that $\rank(X)=r$. For $c_1>0$ sufficiently small, the second term on the right-hand side of~\eqnref{joint_full_rank} is non-negative. This implies that either $X=\Xs$ and $S(X)=\Ss$, or 
	\[2(1-c_1)\leq \frac{\norm{\nabla_X\f_{\mathsf{joint}}(X,S(X))}}{\sigma_r(X)} \leq \frac{1}{\eta}=1,\]
	where the last step is by condition~\eqnref{PGD_SP}---
	but this cannot hold for $c_1>0$ sufficiently small. If instead $\rank(X)<r$, then since $c_1$ is small the inequality~\eqnref{joint_rank_deficient} yields $X=\Xs$ and $S(X)=\Ss$ which is a contradiction since $\Xs$ is full-rank.
	Therefore it follows that $\rank(X)=r$ and $X=\Xs$ and $S(X)=\Ss$,  proving the lemma.
\end{proof}

\begin{proof}[Proof of~\lemref{matrix_completion_local}]
	Let $(A,B)$ be a SOSP of $\g$ with $X=AB^\top \in \mathcal{N}(\Xs)$ and let $\Xs=\As\Bs^\top$ with $\As=\Us\sqrt{\Sigma_\star}$ and $\Bs=\Vs\sqrt{\Sigma_\star}$ where $\Xs=\Us\Sigma_\star\Vs^\top$ is a SVD of $\Xs$.
	For $(A,B)$, let $R\in\R^{r\times r}$ be the best orthogonal rotation matrix to $(\As,\Bs)$, i.e. $R$ is the solution to $\min_{R^\top R=\ident{r}} \fronorm{(A^\top,B^\top)^\top R - (\As^\top,\Bs^\top)^\top}$.
	
	Let $\mathcal{X}=\{X,\Xs\}$. Now to apply our~\thmref{local_opt} to this setting,  we  need to verify that $\f$ satisfies $\alpha$-RSC over $\mathcal{X}$, and that $\norm{\nabla\f(X)}<2\alpha \cdot \sigma_r(X)$  (note $\nabla\f(\Xs)=0$ in our case; $X$ cannot be rank-deficient since $\Xs$ is full-rank and $\fronorm{X-\Xs}\leq 0.1\kappa^{-1}(\Xs)\sigma_r(\Xs)$, see equation~\eqnref{X_sigma_r} below). 
	First, writing $\Delta_A=AR-\As$ and $\Delta_B=BR-\Bs$, we have the decomposition $X-\Xs=\As\Delta_B^\top +\Delta_A\Bs^\top +\Delta_A\Delta_B^\top $. Then, by the work of~\citet[Proof of Lemma 21]{ge2017no}, if $\fronorm{\Delta_A}^2 + \fronorm{\Delta_B}^2\leq \sigma_r(\Xs)/40$, then by our choice of $p$, we have with high probability
	\begin{align*}
		\frac{1}{2p}\fronorm{\Pr{\Omega}{X-\Xs}}^2 &= \frac{1}{2p}\fronorm{\Pr{\Omega}{\As\Delta_B^\top + \Delta_A\Bs^\top + \Delta_A\Delta_B^\top }}^2 \\
		&\geq \frac{1}{4}\sigma_r(\Xs)(\fronorm{\Delta_A}^2 + \fronorm{\Delta_B}^2).
	\end{align*} 
	Similarly, we can prove that with high probability
	\[\frac{1}{2p}\fronorm{\Pr{\Omega}{X-\Xs}}^2  \leq \frac{3}{2}\sigma_1(\Xs)(\fronorm{\Delta_A}^2 + \fronorm{\Delta_B}^2).\]
	Furthermore, by~\citet[Lemma 5.3]{tu2015low}, we can deduce that $0.35\sigma_1^{-1}(\Xs)\fronorm{X-\Xs}^2\leq\fronorm{\Delta_A}^2 + \fronorm{\Delta_B}^2$ while by~\citet[Lemma 5.4]{tu2015low} and~\citet[Lemma 4]{Zheng2016Convergence}, together with~\eqnref{mu_incoherence2} that $A^\top A = B^\top B$, we have $\fronorm{\Delta_A}^2 + \fronorm{\Delta_B}^2\leq 2.5 \sigma_r^{-1}(\Xs)\fronorm{X-\Xs}^2$. Putting everything together, it follows that
	\[0.08\kappa^{-1}(\Xs) \fronorm{X-\Xs}^2 \leq  \frac{1}{2p}\fronorm{\Pr{\Omega}{X-\Xs}}^2 \leq 3.75\kappa(\Xs)\fronorm{X-\Xs}^2, \]
	whenever $\fronorm{X-\Xs}^2\leq 0.01\sigma_r^2(\Xs)$. In particular this proves restricted strong convexity over $\mathcal{X}=\{X,\Xs\}$, with $\alpha=0.08\kappa^{-1}(\Xs)$.
	
	Next, to show  $\norm{\nabla\f(X)}<2\alpha \cdot \sigma_r(X)$, it suffices to show
	\begin{equation}\label{eqn:goal}
		\frac{1}{p}\norm{\Pr{\Omega}{X-\Xs}} < \frac{0.16\sigma_r(X)}{\kappa(\Xs)}.
	\end{equation}
	We denote $\tPr{\Omega}{L} = \frac{1}{p}\Pr{\Omega}{L} - L$. Then we split the left hand side of equation~\eqnref{goal} into two terms
	\[\frac{1}{p}\norm{\Pr{\Omega}{X-\Xs}} \leq \norm{X-\Xs} +\norm{\tPr{\Omega}{X-\Xs}}.\]
	The first term is bounded by our assumption as $\norm{X-\Xs}\leq 0.1\kappa^{-1}(\Xs)\sigma_r(\Xs)$. The second term is upper bounded by
	\begin{align*}
		\norm{\tPr{\Omega}{X-\Xs}} &\leq \norm{\tPr{\Omega}{\Delta_A (BR)^\top}} + \norm{\tPr{\Omega}{\As\Delta_B^\top}}\\
		&\leq \norm{\tPr{\Omega}{11^\top}} \norm{\Delta_A}_{2,\infty} \norm{BR}_{2,\infty} + \norm{\tPr{\Omega}{11^\top}} \norm{\As}_{2,\infty} \norm{\Delta_B}_{2,\infty}\\
		&\lesssim \frac{\sqrt{m+n}}{\sqrt{p}} \norm{\Delta_A}_{2,\infty} \norm{BR}_{2,\infty} + \frac{\sqrt{m+n}}{\sqrt{p}}\norm{\As}_{2,\infty} \norm{\Delta_B}_{2,\infty}\\
		&\leq c_3\frac{\sqrt{m+n}}{\sqrt{p}}\frac{ \sigma_1(\Xs) \mu r}{\sqrt{mn}}  \leq 0.04 \kappa^{-1}(\Xs)\sigma_r(\Xs),
	\end{align*}
	with high probability. Here the first step uses the identity $X-\Xs=\Delta_A (BR)^\top + \As \Delta_B^\top$ together with triangle inequality; the  second and the third steps are respectively due to~\citet[Lemma 4.5]{chen2017memory} and~\citet[Lemma 3.2]{keshavan2010matrix}; the fourth step uses the fact that $\norm{BR}_{2,\infty}=\norm{B}_{2,\infty}$ for orthogonal matrix $R$, as well as the incoherence assumption~\eqnref{mu_incoherence2}; and the last step holds since $pmn\geq \mathcal{O}\left(\mu^2 r^2\kappa^4(\Xs) (m+n)\log (m+n) \right)$.  The right hand side of~\eqnref{goal} is lower bounded as follows:
	\begin{equation}\label{eqn:X_sigma_r}\sigma_r(X) \geq \sigma_r(\Xs) - \norm{X-\Xs}\geq 0.9\sigma_r(\Xs).\end{equation}
	Combining the above bounds, it is straightforward to see that \eqnref{goal} holds. Finally, applying~\thmref{local_opt} proves the desired result.
	
\end{proof}

\section{Nondegenerate case for robust PCA}\label{app:rpca}
Consider the robust PCA problem given in~\secref{rpca}, and we additionally assume that the sparse component $\Ss$ has bounded entries, i.e., $\norm{\Ss}_\infty\leq c\frac{\mu r \sigma_1(\Xs)}{\sqrt{mn}}$ for some constant $c>0$. As pointed out by~\citet{ge2017no}, this requirement is not without loss of generality, because any $\mu$-incoherent matrix $\Xs$ has maximum entries bounded by $\frac{\mu r \sigma_1(\Xs)}{\sqrt{mn}}$. Here we consider the following sparsity constraint set, 
\[\bar{\Sset}=\{S\in\R^{m\times n}: \norm{S_{j}}_0\leq  \norm{\Ss{}_{j}}_0=s \textnormal{ and } \norm{S}_\infty\leq \frac{c\mu r\sigma_1(\Xs)}{\sqrt{mn}} \}. \]
Then the following statement holds: if $m\geq \mathcal{O}(s\cdot \mu^2 r^2\kappa^2)$,
\begin{equation*}\label{eqn:sval_lemma}\textnormal{For any $S\in\bar{\Sset}$, $\sigma_{r+1}(\Ds - S) < \sigma_r(\Ds-S)$.}\end{equation*}
To see why, we first bound $\norm{S-\Ss}$. Writing $\Delta_S = S-\Ss$, we can calculate
\begin{multline*}
	\norm{\Delta_S} = \sup_{\norm{x}_2=\norm{y}_2=1}x^\top \Delta_S y =  \sup_{\norm{x}_2=\norm{y}_2=1}\sum_{j=1}^n y_j \cdot \Delta_{S,j}^\top x \leq \sup_{\norm{x}_2=\norm{y}_2=1} \sum_{j=1}^n|y_j|\norm{\Delta_{S,j}}_2\norm{x}_2\\
	\leq\sup_{\norm{y}_2=1} \norm{y}_2 \sqrt{ \sum_{j=1}^n\norm{\Delta_{S,j}}_2^2} \leq \sqrt{\frac{s}{m}}\cdot c_2\mu r\sigma_1(\Xs),
\end{multline*}
where $\Delta_{S,j}$ denotes the $j$-th column of $\Delta_S$, and the last step holds since $\norm{\Delta_{S,j}}_2\leq \sqrt{\frac{s\cdot c_2^2 \mu^2 r^2 \sigma_1^2(\Xs)}{mn}}$ for some $c_2>0$. Then if $m\gtrsim s\cdot \mu^2r^2\kappa^2$, we can find $c_3>0$ such that $\norm{\Delta_S}\leq c_3\sigma_r(\Xs)$. Therefore, given the data matrix $\Ds = \Xs +\Ss$, we get
\[ \sigma_r(\Ds -S)\geq \sigma_r(\Xs) - \norm{S-\Ss}\geq (1-c_3)\sigma_r(\Xs) \textnormal{ and } \sigma_{r+1}(\Ds-S)\leq c_3\sigma(\Xs). \]
Taking $c_3<1/2$ then proves the desired result.

\end{document}